\documentclass[
final
, nomarks
]{dmtcs-episciences}

\usepackage[utf8]{inputenc}
\usepackage{subfigure}

\usepackage{amsmath, amsfonts, amssymb}
\usepackage{amsthm, mathtools}
\makeatletter
\let\amsthm@proof\proof
\let\amsthm@endproof\endproof
\AtBeginDocument{%
  \let\proof\amsthm@proof
  \let\endproof\amsthm@endproof
}
\makeatother
\usepackage{diagbox}
\usepackage{enumitem}
\usepackage{epsfig, epstopdf, color}
\usepackage{fancyhdr}
\usepackage[margin=0.75cm]{caption}
\usepackage[labelfont=bf]{caption}
\usepackage[table]{xcolor}
\usepackage{tikz}

\usepackage{hyperref}

\usetikzlibrary{arrows.meta}
\usetikzlibrary{automata}
\usetikzlibrary{decorations.pathreplacing,quotes,angles,decorations.markings,arrows}
\usetikzlibrary{patterns}

\DeclareMathOperator{\Av}{Av}
\newcommand\orange[1]{\textcolor{orange}{\bf{#1}}}
\newcommand{\card}[1]{\left|#1\right|}
\newcommand\cB{\mathcal{B}}

\newcommand{\dash}{\text{-}}
\newcommand{\ds}{\displaystyle}
\newcommand{\EF}{\E F}
\DeclareMathOperator{\Em}{Em}
\newcommand{\E}{\mathbb{E}}
\newcommand{\ET}{\E T}
\newcommand{\EI}{\E I}
\newcommand\Var{\mathbb{V}}
\newcommand{\F}{F}

\renewcommand{\S}{\mathcal S}
\DeclareMathOperator{\st}{st}
\newcommand{\T}{T}

\newcommand{\I}{I}
\newcommand{\VT}{\Var T}
\newcommand{\X}{\mathbf{X}}
\newcommand{\comp}[1]{\overline{#1}}
\newcommand{\tie}{\cellcolor{blue}$\equiv$}
\newcommand{\win}{\cellcolor{cyan}$>$}
\newcommand{\loss}{\cellcolor{magenta}$<$}
\newcommand{\rot}[1]{\rotatebox{90}{\textcolor{magenta}{#1}}}
\newcommand{\cya}[1]{\textcolor{cyan}{#1}}
\newcommand{\ul}{\underline}
\newcommand{\ol}{\overline}

\newtheorem{theorem}{Theorem}[section]
\newtheorem{conjecture}[theorem]{Conjecture}  
\newtheorem{corollary}[theorem]{Corollary}
\newtheorem{lemma}[theorem]{Lemma}
\newtheorem{observation}[theorem]{Observation}

\newtheorem{proposition}[theorem]{Proposition}
\newtheorem{problem}[theorem]{Problem}  

\theoremstyle{definition}

\newtheorem*{example}{Example}

\author{Sergi Elizalde\thanks{Partially supported by Simons Collaboration Grant \#929653.}
  \and Yixin Lin}
\title{Penney's game for permutations}
\affiliation{
  Department of Mathematics, Dartmouth College, Hanover, NH, US
  }
\begin{document}
\publicationdata{vol. 28:1, Permutation Patterns 2025}{2026}{3}{10.46298/dmtcs.16361}{2025-08-18; 2025-08-18; 2026-04-29}{2026-04-30}
\maketitle
\begin{abstract}
We consider a permutation analogue of Penney's game for words. Two players, in order, each choose a permutation of length $k\ge3$; then a sequence of independent random values from a continuous distribution is generated, until the relative order of the last $k$ numbers coincides with one of the chosen permutations, making that player the winner.
 We compute the winning probabilities for all pairs of permutations of length $3$ and some pairs of length $4$, showing that, as in the original version for words, the game is non-transitive. Our proofs introduce new bijections for consecutive patterns in permutations. We also give some formulas to compute the winning probabilities more generally, and conjecture a winning strategy for the second player when $k$ is arbitrary.
\end{abstract}

\section{Introduction}

In Penney's game, also known as Penney Ante~\cite{Penney}, player~A selects a binary word of length $k\ge3$, then player~B selects another binary word of the same length. A fair coin is tossed repeatedly to determine a random binary word, until one of the players' words appears as a consecutive subword, making that player the winner.
It is known that, for any word picked by player~A, player~B can always pick a word that will be more likely to appear first. This property makes Penney's game a non-transitive game.

The exact odds of winning can be computed using Conway's formula~\cite{Conway}, in terms of the overlaps of the two words with themselves and with each other. The same method applies to words over any finite alphabet. 
Guibas and Odlyzko provided a proof of Conway's formula using generating functions~\cite{GuibasOdlyzko}, as well as a winning strategy for player~B, by describing how to choose a word that will be more likely to appear first than player~A's word.
Another proof of Conway's formula, using induction, was given by Collings~\cite{CollingsProof} soon after.
Extending Guibas and Odlyzko's approach, Felix~\cite{Felix} characterized the words that maximize the winning probabilty for player~B, given the word chosen by player~A, over an alphabet of arbitrary size.

In this paper, we consider an analogue of Penney's game for permutations. Now player~A selects a permutation of length $k\ge3$, then player~B selects another permutation of the same length. Then a sequence of independent random numbers is generated until the relative order of the $k$ most recent numbers agrees with one of the players' permutations, making that player the winner.

More specifically, let $\X=X_1,X_2,\dots$ be a sequence of i.i.d.\ continuous random variables; for example, we can assume without loss of generality that they have a uniform distribution in $[0,1]$. Let $\S_k$ denote the set of permutations of $[k]\coloneqq\{1,2,\dots,k\}$.
A {\em consecutive occurrence} of a permutation (often called a {\em pattern}) $\sigma\in\S_k$ is a subsequence $X_i,X_{i+1},\dots,X_{i+k-1}$ whose entries are in the same relative order as $\sigma_1,\sigma_2,\dots,\sigma_k$, that is, $\st(X_iX_{i+1} \dots X_{i+k-1})=\sigma$, where $\st$ denotes the standardization operation that replaces the smallest entry with $1$, the second smallest with $2$, etc.
After player~A chooses a permutation $\sigma$ and player~B chooses a permutation $\tau$, the random variables $X_1,X_2,\dots$ are drawn until a consecutive occurrence of $\sigma$ or $\tau$ appears, which determines the winner.
One can consider a more general version where $\sigma$ and $\tau$ may have different lengths. Using continuous random variables ensures that the event $X_i=X_j$ (where $i\neq j$) has probability zero.

In Section~\ref{Sec:ExpectedTime}, we provide a simple expression for the expected time until a given pattern appears for the first time.
In Section~\ref{Sec:GeneralProbability}, we study the probability that one pattern appears before another. We give a formula for this probability in terms of expectations, inspired in Conway's formula but significantly less practical, and another one that relates to permutation enumeration. 
We then explicitly compute this probability in some special cases, and we describe several families of pairs of patterns for which this probability is $1/2$, i.e., both patterns have equal probability of appearing first.
In Section~\ref{Sec:Length3}, we provide a complete table for the probabilities of seeing one pattern before another when both patterns have length~3, computed using bijections and recurrences. 
In Section~\ref{Sec:Length4}, we consider some pairs patterns of length~4, including the curious case of the patterns $2134$ and $3241$,
each of which has equal probability of appearing before the other, but for non-obvious reasons. Our proof of this fact introduces some new bijections for consecutive patterns in permutations, in combination with an inclusion-exclusion argument.
In Section~\ref{Sec:Conjecture}, we conclude that Penney's game for permutations is non-transitive. We conjecture that for any permutation $\sigma_1\dots\sigma_{k-1}\sigma_k\in\S_k$, where $k\ge3$, the permutation $\sigma_k\sigma_1\dots\sigma_{k-1}$ is more likely to appear first, which would provide
a winning strategy for player~B for permutations of arbitrary length.


\section{The first occurrence of a pattern}
\label{Sec:ExpectedTime}

In this section we analyze how long it takes for given pattern to appear for the first time in the random sequence $\X$. The analogous problem for words over a finite alphabet was studied by Nielsen~\cite{PNielsen}, motivated by frame synchronization methods in digital communication systems.

Given $\pi\in\S_n$ and $\sigma\in\S_k$, an occurrence of $\sigma$ in $\pi$ as a consecutive pattern is a subsequence of adjacent entries of $\pi$ that are in the same relative order as the entries of $\sigma$. For example, occurrences of $132$ and $231$ are often called peaks, and occurrences of $321$ are called double descents. If $\pi$ does not contain any occurrences of $\sigma$, we say that $\pi$ {\em avoids} $\sigma$. See~\cite{Elizalde_survey} for a survey of consecutive patterns in permutations.
Unless otherwise stated, all the notions of patterns, occurrences, containment and avoidance in this paper refer to the consecutive case. However, we will see later that the definitions and results in this section extend immediately to so-called {\em vincular patterns}.

Let $\alpha_n(\sigma)$ be the number of permutations in $\S_n$ that avoid $\sigma$ as a consecutive pattern, and denote the corresponding exponential generating function by \[P_\sigma(z)=\sum_{n\ge0}\alpha_n(\sigma)\frac{z^n}{n!}.\]  Define the random variable $\T_\sigma$ as the smallest $j$ such that $\textbf X=X_1,\dots,X_j$ contains a consecutive occurrence of $\sigma$. In other words, $\T_\sigma$ is the number of random draws until we see $\sigma$ for the first time. Our first result is that the expectation $\ET_\sigma$ and the variance $\VT_\sigma$ have surprisingly simple expressions in terms of this generating function.

\begin{theorem}\label{thm:ET}
	For every $\sigma$, \begin{align*} 
		\ET_\sigma&=P_\sigma(1),\\
		\VT_\sigma&=2 P'_\sigma(1)+ P_\sigma(1)- P_\sigma(1)^2.
	\end{align*}
\end{theorem}

\begin{proof}
	Let $A_j$ be the event that $X_1,X_2,\dots,X_j$ avoids the pattern $\sigma$. Then \[\Pr(T_\sigma\ge i)=\Pr(A_{i-1})=\frac{\alpha_{i-1}(\sigma)}{(i-1)!},\] from where 
 \[\ET_\sigma=\sum_{i\ge1}\Pr(T_\sigma\ge i)=\sum_{i\ge1}\frac{\alpha_{i-1}(\sigma)}{(i-1)!}=P_\sigma(1).\]
Similarly,
\[\ET_\sigma^2=\sum_{i\ge 1} (2i-1) \Pr(T_\sigma\ge i)=\sum_{n\ge0} (2n+1) \frac{\alpha_n(\sigma)}{n!}
=2 P'_\sigma(1)+ P_\sigma(1),\]
and so 
\[\VT_\sigma =\ET_\sigma^2-(\ET_\sigma)^2=2 P'_\sigma(1)+ P_\sigma(1)-P_\sigma(1)^2.\qedhere\]
\end{proof}

Expressions for the generating function $P_\sigma(z)$ for various patterns $\sigma$ have been obtained by Elizalde and Noy~\cite{ElizaldeNoy4, ElizaldeNoy5}. In particular, for patterns of length~3, \cite[Theorem 4.1]{ElizaldeNoy4} gives
\begin{equation}\label{eq:P123}
P_{123}(z)=\frac{\sqrt{3}}{2}\frac{e^{z/2}}{\cos(\frac{\sqrt{3}}{2}z+\frac{\pi}{6})},\qquad
P_{132}(z)=\frac{1}{1-\int_0^z e^{-t^2/2}\,dt}.
\end{equation}
For an arbitrary monotone pattern, the following formula appears in~\cite{DavidBarton}:
\begin{equation}\label{eq:Pmonotone} P_{12\dots k}(z)=\left(\sum_{j\ge0}\frac{z^{jk}}{(jk)!}-\sum_{j\ge0}\frac{z^{jk+1}}{(jk+1)!}\right)^{-1}.
\end{equation}

Using these expressions, it follows from Theorem~\ref{thm:ET} that, for example, 
\begin{align}\label{eq:ET123}
   & \ET_{12}=e, && \ET_{123}=\frac{\sqrt{3e}}{2\cos(\frac{\sqrt{3}}{2}+\frac{\pi}{6})}\approx 7.924, && \ET_{132}=\frac{1}{1-\int_0^1 e^{-t^2/2} dt}\approx 6.926,\\
   \nonumber
   & \VT_{12}=3e-e^2, && \VT_{123}\approx 27.981, && \VT_{132}\approx 17.148.
\end{align}

Recall that two  patterns $\sigma$ and $\tau$ are said to be {\em Wilf-equivalent} if
$P_\sigma(z)=P_\tau(z)$.

\begin{corollary}\label{cor:ET}
	If $\sigma$ and $\tau$ are Wilf-equivalent, then the random variables $\T_\sigma$ and $\T_\tau$ have the same distribution. In particular, $\ET_\sigma=\ET_\tau$ and $\VT_\sigma=\VT_\tau$.
\end{corollary}

\begin{proof}
    The fact that $\ET_\sigma=\ET_\tau$ and $\VT_\sigma=\VT_\tau$ follows immediately from Theorem~\ref{thm:ET}. More generally, we have $$\Pr(T_\sigma=i)=\Pr(T_\sigma\ge i)-\Pr(T_\sigma\ge i-1)=
    \frac{\alpha_{i-1}(\sigma)}{(i-1)!}-\frac{\alpha_{i-2}(\sigma)}{(i-2)!}.
    $$
    Since $\sigma$ and $\tau$ are Wilf-equivalent, we have $\alpha_n(\sigma)=\alpha_n(\tau)$ for all $n$, and so the above expression equals $\Pr(T_\tau=i)$.
\end{proof}

To study the analogue for permutations of Penney's game, as described above, consecutive patterns provide the suitable setting.
However, the definition of $\T_\sigma$, as well as Theorem~\ref{thm:ET} and Corollary~\ref{cor:ET}, can easily be extended to the case where $\sigma$ is a vincular or a classical pattern. 
A vincular pattern (called generalized pattern in~\cite{BabsonSteingrimsson}) is a permutation where dashes may be inserted between some of its entries, indicating that the corresponding positions do not need to be adjacent in an occurrence of the pattern in a permutation. For example, an occurrence of the vincular pattern $1\dash32$ in $\pi$ is any subsequence $\pi_i\pi_j\pi_{j+1}$ with $i<j$ and $\pi_i<\pi_{j+1}<\pi_j$. Vincular patterns with dashes in all positions are called classical patterns, which have no adjacency requirement, and appear most frequently in the combinatorics literature~\cite{SimionSchmidt}.

In its extended version, Theorem~\ref{thm:ET} can be used compute that the expected number of draws before the first occurrence of the vincular pattern $1\dash23$ is 
\[\ET_{1\dash23}=\sum_{n\ge0} \frac{B_n}{n!}=e^{e-1}\approx5.575,\]
where $\alpha_n(1\dash23)=B_n$ is the $n$th Bell number~\cite{Claesson}. And if $\sigma$ is any classical pattern of length $3$, this expected number is 
\[\ET_{\sigma} = \sum_{n\ge0} \frac{C_n}{n!}=e^2I_2(2)\approx5.091,\] 
where $\alpha_n(\sigma)=C_n$ is the $n$th Catalan number~\cite{SimionSchmidt}, and $I_\nu(z)$ denotes the modified Bessel function of the first kind. 

For any vincular pattern $\sigma$, the value $\ET_\sigma$ gives a measure of how easy it is to avoid $\sigma$ in a random permutation. Unlike other measures 
commonly used in the literature, such as $\lim_n \alpha_n(\sigma)^{1/n}$ for classical patterns and $\lim_n \left(\frac{\alpha_n(\sigma)}{n!}\right)^{1/n}$ for consecutive patterns, the advantage of using $\ET_\sigma$ is that it can be defined uniformly for all vincular patterns, providing always a positive number, which allows us to compare different patterns.

The analogue of Theorem~\ref{thm:ET} for words, where an $m$-sided die is repeatedly rolled until a given word $w=w_1w_2\dots w_k\in[m]^k$ appears, was obtained by Nielsen in~\cite{PNielsen}. If $\T_w$ is the random variable counting the number of rolls until the first appearance of $w$, it is known that 
\[\ET_w=\sum_{i\in\cB_w} m^i,\]
where  $\cB_w=\{i\in[k]: w_1\dots w_i=w_{k-i+1}\dots w_{k}\}$ is the {\em bifix set} of $w$. It is interesting to note that $\ET_w$ is always an integer, unlike in the case of permutations.


\section{The probability of seeing one pattern before another}
\label{Sec:GeneralProbability}

In the case of words, Conway gave a beautiful formula for the probability that an $m$-ary word appears before another one in a sequence of tosses of an $m$-sided die, see Gardner's account~\cite{Conway}. Given two words $w\in[m]^k$ and $v\in[m]^\ell$ such that neither of them is a consecutive subword of the other, let
$$\cB_{v,w}=\{i\in[\min\{k,\ell\}]: w_1\dots w_i=v_{\ell-i+1}\dots v_{\ell}\},$$ and $$b(v,w)=\sum_{i\in\cB_{v,w}} m^{i-1}.$$ Note that $\cB_{w,w}=\cB_w$.

\begin{theorem}[Conway~\cite{Conway}]\label{thm:Conway}
With the above definitions, the probability that $w$ appears before $v$ in a random sequence is
$$\frac{b(v,v)-b(v,w)}{b(v,v)-b(v,w)+b(w,w)-b(w,v)}.$$
\end{theorem}

For example, the probability that the binary word $w=100$ appears before $v=000$ is
$$\frac{7-0}{7-0+4-3}=\frac{7}{8}.$$

We now turn to the case of permutations. Given two patterns $\sigma\in\S_k$ and $\tau\in\S_\ell$ such that neither of them contains a consecutive occurrence of the other --- we call these {\em incomparable} patterns --- denote by $\Pr(\sigma\prec\tau)$ the probability that $\sigma$ appears before $\tau$ in the random sequence $\X$. There seems to be no simple expression for $\Pr(\sigma\prec\tau)$ for arbitrary $\sigma$ and $\tau$, and we will see in Section~\ref{Sec:Length3} that this value is not always a rational number. 

We start with some observations that follow trivially by symmetry. These will reduce the number of cases that we have to consider in Sections~\ref{Sec:Length3} and~\ref{Sec:Length4}. For $\sigma\in\S_k$, define its complement
$\comp{\sigma}\in\S_k$  by $\comp{\sigma}_i=k+1-\sigma_i$ for $1\le i\le k$.

\begin{observation}\label{obs} Let $\sigma$ and $\tau$ be two incomparable patterns. Then
\begin{enumerate}
	\item $\Pr(\tau\prec\sigma)=1-\Pr(\sigma\prec\tau)$,
	\item $\Pr(\sigma\prec\tau)=\Pr(\comp{\sigma}\prec\comp{\tau})$,
	\item $\Pr(\sigma\prec\comp{\sigma})=\frac12$.
\end{enumerate}
\end{observation}

Denoting the increasing monotone pattern by $\iota_k=12\dots k$, it is easy to see that
\begin{equation}\label{eq:P1k<21}
\Pr(\iota_k\prec 21)=\frac{1}{k!},
\end{equation}
since the only way for the pattern $\iota_k$ to appear first is if the initial $k$ random values are increasing.
In this section we give some general formulas for $\Pr(\sigma\prec\tau)$ that will help us compute 
these probabilities in specific cases.

In the case of words, Collings' proof~\cite{CollingsProof} of Conway's formula relies on the expected number of additional tosses to see one word assuming that the sequence of tosses starts with the other word.
For two incomparable permutations $\sigma\in\S_k$ and $\tau\in\S_\ell$, there are multiple ways to define an analogue of this notion. Next we define two different random variables, which we denote by $\I_{\sigma\to\tau}$ and $\F_{\sigma\to\tau}$.

Define $\I_{\sigma\to\tau}$ as the smallest $j$ such that $X_1,\dots,X_{k+j}$ contains a consecutive occurrence of $\tau$, conditioning on the fact that $X_1,\dots,X_k$ is an occurrence of $\sigma$. In other words, $\I_{\sigma\to\tau}$ is the number of further draws needed to see $\tau$, assuming that $\sigma$ occurs at the beginning of~$\X$ (the letter $\I$ stands for {\em initial}).
Disregarding the requirement that the two patterns are incomparable, we could define $\I_{\sigma\to\sigma}$ as the number of 
further draws needed to see a second occurrence of $\sigma$ (i.e., not counting the initial one). It is not hard to see that $\EI_{\sigma\to\sigma}=k!$ for any $\sigma\in\S_k$, since the expected proportion of occurrences of $\sigma$ in a long random sequence is $1/k!$, so each one is separated from the next by $k!$ on average. 

In general, for incomparable $\sigma\in\S_k$ and $\tau\in\S_\ell$, we have
\begin{align}\nonumber
	\EI_{\sigma\to\tau}&=\sum_{i\ge0}\Pr(\I_{\sigma\to\tau}\ge i)=\sum_{n\ge k}\frac{\#\{\pi\in\S_{n} \text{ that start with }\sigma\text{ and avoid }\tau\}}{\#\{\pi\in\S_{n} \text{ that start with }\sigma\}}\\ \label{eq:EU}
 &=k!\sum_{n\ge k}\frac{\#\{\pi\in\S_{n} \text{ that start with }\sigma\text{ and avoid }\tau\}}{n!},
\end{align}
noting that the number of $\pi\in\S_n$ that start with $\sigma$ (meaning that $\pi_1\pi_2\dots\pi_k$ is an occurrence of $\sigma$)  is simply $n!/k!$. Let us show that, in some cases, it is possible to compute $\EI_{\sigma\to\tau}$ using this formula.

\begin{proposition}\label{prop:EU1k21}	For all $k\ge2$,
\[\EI_{\iota_k\to21}=k!\left(e-\sum_{i=0}^{k-1}\frac{1}{i!}\right).\]
\end{proposition}

\begin{proof}
Using equation~\eqref{eq:EU} with $\sigma=\iota_k$ and $\tau=21$, and noting that the only permutation in $\S_{n}$ that avoids $21$ is $12\dots n$, we have
\[
		\EI_{\iota_k\to21}
  =k!\sum_{n\ge k}\frac{1}{n!}=k!\left(e-\sum_{i=0}^{k-1}\frac{1}{i!}\right).\qedhere
\]
\end{proof}

In the next example, we compute $\EI_{\sigma\to\tau}$ by adapting the cluster method from~\cite{ElizaldeNoy5}.

\begin{theorem}\label{thm:EU21123}
\[\EI_{21\to123}=\sqrt3\tan\left(\frac{\sqrt3}{2}+\frac\pi6\right)-3\approx6.4554.\]
\end{theorem}

\begin{proof}
Let $P_{123}(z)$ be exponential generating function for $123$-avoiding permutations, and let $D(z)$ be the exponential generating function for $123$-avoiding permutations that start with the pattern $21$. 

Every non-empty $123$-avoiding permutation $\pi$ can be decomposed uniquely as $\pi=\alpha 1\beta$, where $\st(\alpha)$ is an arbitrary $123$-avoiding permutation, and $\st(\beta)$ is an arbitrary $123$-avoiding permutation that starts with the pattern $21$, or otherwise has length $0$ or $1$. It follows that
    \[P_{123}'(z)=P_{123}(z)(D(z)+1+z),\] and so  
    \[D(z)=\frac{P'_{123}(z)}{P_{123}(z)}-1-z.\]

Equation~\eqref{eq:EU} now implies that
\[\EI_{21\to123}=2D(1)=2\frac{P'_{123}(1)}{P_{123}(1)}-4=\sqrt3\tan\left(\frac{\sqrt3}{2}+\frac\pi6\right)-3,
\]
using the expression for $P_{123}(z)$ given in equation~\eqref{eq:P123}.
\end{proof}

Unfortunately, finding $\EI_{\sigma\to\tau}$ and $\EI_{\tau\to\sigma}$ does not seem to help in computing $\Pr(\sigma\prec\tau)$.
Instead, one can define $\F_{\sigma\to\tau}$ as the number of further draws needed to see the pattern $\tau$ after the first occurrence of $\sigma$, assuming that $\sigma$ occurs before $\tau$ in $\X$ (the letter $\F$ stands for {\em first}). The variables $\I_{\sigma\to\tau}$ and $\F_{\sigma\to\tau}$ are different in general; for example, we will see later
that $\EI_{21\to123}\approx6.4554$ but $\EF_{21\to123}\approx 6.5092$.
The following result shows that the expectation of the random variables $\F_{\sigma\to\tau}$ is closely related to the probability that one pattern occurs before another, in analogy to Collings' formula for words~\cite{CollingsProof}.

\begin{theorem}\label{thm:precEF}
	For any two incomparable patterns $\sigma$ and $\tau$, \[\Pr(\sigma\prec\tau)=\frac{\EF_{\tau\to\sigma}+\ET_\tau-\ET_\sigma}{\EF_{\tau\to\sigma}+\EF_{\sigma\to\tau}}.\]
\end{theorem}

\begin{proof}
The expected (signed) difference between the ending positions of the first occurrence of $\sigma$ and the first occurrence of $\tau$ is $\ET_\sigma-\ET_\tau$. 

On those occasions when $\tau$ precedes $\sigma$, the average further draws to get $\sigma$ from $\tau$ is $\EF_{\tau\to\sigma}$. Similarly, on those occasions when $\sigma$ precedes $\tau$, the average further draws to get $\tau$ from $\sigma$ is $\EF_{\sigma\to\tau}$. Therefore, 
\[
	\ET_\sigma-\ET_\tau=\Pr(\tau\prec\sigma)\,\EF_{\tau\to\sigma}-\Pr(\sigma\prec\tau)\,\EF_{\sigma\to\tau}.
 \]
Using part 1 of Observation~\ref{obs} and solving for $\Pr(\sigma\prec\tau)$, we get the stated formula.
\end{proof}

The downside of Theorem~\ref{thm:precEF} is that obtaining expressions for $\EF_{\sigma\to\tau}$ for arbitrary patterns is difficult. The next result shows that it can be done in some cases.

\begin{proposition}\label{prop:EFiota21} For all $k\ge2$,
	\begin{align*}\EF_{\iota_k\to21}&=k!\left(e-\sum_{i=0}^{k-1}\frac{1}{i!}\right),\\
 \EF_{21\to\iota_k}&=\frac{k!}{k!-1}\left(
 P_{\iota_k}(1)
 -\sum_{i=0}^{k-1}\frac{1}{i!}\right),
 \end{align*}
 where $P_{\iota_k}(1)$ is given by equation~\eqref{eq:Pmonotone}.
\end{proposition}

\begin{proof}
The first formula follows from Proposition~\ref{prop:EU1k21}, using the fact that $\F_{\iota_k\to21}=\I_{\iota_k\to21}$.
This is because if $\iota_k$ occurs before $21$ in $\X$, then $\X$ must start with $\iota_k$.

To compute $\EF_{21\to\iota_k}$, we note that the condition that $21$ appears before $\iota_k$ in $\X$ is equivalent to the fact that the first $k$ entries of $\X$ form a pattern in $\S_k\setminus\{\iota_k\}$. Suppose that the pattern formed by these entries is $\sigma\in\S_k\setminus\{\iota_k\}$, and let $f(\sigma)$ be the first descent of $\sigma$, i.e., the smallest $j$ such that $\sigma_j>\sigma_{j+1}$. In this case, the variables $F_{21\to\iota_k}$ and $I_{\sigma\to\iota_k}$ differ by $k-(j+1)$, since 
both count further draws until we see $\iota_k$, starting after $X_{j+1}$ and after $X_{k}$, respectively.
Summing over the possible patterns, it follows that
\begin{equation}\label{eq:EF211k} \EF_{21\to\iota_k}=\frac{1}{k!-1}\sum_{j=1}^{k-1}\sum_{\substack{\sigma\in\S_k\setminus\{\iota_k\}\\ f(\sigma)=j}}(\EI_{\sigma\to\iota_k}+k-j-1).
\end{equation}

Similarly, one can express $\ET_{\iota_k}$ by summing over the possible patterns $\sigma\in\S_k$ formed by the first $k$ entries of $\X$. If $\sigma=\iota_k$, then $\T_{\iota_k}=k$; otherwise, $\T_{\iota_k}=\I_{\sigma\to\iota_k}+k$, since $\I_{\sigma\to\iota_k}$ counts the further draws needed to see $\iota_k$ after the initial occurrence of $\sigma$. Thus,
\begin{equation}\label{eq:ET1k} \ET_{\iota_k}=\frac{1}{k!}\left(k+\sum_{\sigma\in\S_k\setminus\{\iota_k\}}(\EI_{\sigma\to\iota_k}+k)\right). \end{equation}

Comparing equations~\eqref{eq:EF211k} and~\eqref{eq:ET1k} ,
\begin{equation}\label{eq:compareETF}
k!\,\ET_{\iota_k}-(k!-1)\,\EF_{21\to\iota_k} =k+\sum_{j=1}^{k-1}\sum_{\substack{\sigma\in\S_k\setminus\{\iota_k\}\\ f(\sigma)=j}}(j+1) = k+\sum_{j=1}^{k-1} \frac{k!}{(j-1)!}.
\end{equation}
In the last equality, we used the fact that, in order to have $f(\sigma)=j$, the first $j+1$ entries of $\sigma$ must satisfy $\sigma_1<\sigma_2<\dots<\sigma_j>\sigma_{j+1}$. Of the $(j+1)!$ permutations of the first $j+1$ entries, exactly $j$ satisfy this condition. Thus, 
$$\card{\{\sigma\in\S_k\setminus\{\iota_k\}:f(\sigma)=j\}}=j\,\frac{k!}{(j+1)!}.$$

Equation~\eqref{eq:compareETF} can be written as 
$$\ET_{\iota_k}=\left(1-\frac{1}{k!}\right)\EF_{21\to\iota_k}+\sum_{i=0}^{k-1} \frac{1}{i!}.$$
Solving for $\EF_{21\to\iota_k}$ and using Theorem~\ref{thm:ET}, we obtain the stated expression.
\end{proof}

For example, for $k=3$, Proposition~\ref{prop:EFiota21} gives 
$$\EF_{123\to21}=6e-15\approx1.3097,$$
and, using equation~\eqref{eq:P123} for $P_{123}(z)$, 
$$\EF_{21\to123}=\frac{3\sqrt{3e}}{5\cos\left(\frac{\sqrt3}{2}+\frac\pi6\right)}-3\approx 6.5092.$$

Plugging the above expressions for $\EF_{123\to21}$ and $\EF_{21\to123}$, along with the expressions for $\ET_{21}=\ET_{12}$ and $\ET_{123}$ from equation~\eqref{eq:ET123}, into Theorem~\ref{thm:precEF}, we deduce, after simplifying, that
$$\Pr(123\prec21)=\frac16,$$
which agrees with equation~\eqref{eq:P1k<21}.

Another approach to computing the probability that $\sigma$ appears before $\tau$ is to count the relevant permutations. We will use this approach throughout the paper.

Denote by $\Av_n^\sigma(\sigma,\tau)$ the set of permutations in $\S_n$ that end with an occurrence of $\sigma$, and avoid both $\sigma$ and $\tau$ elsewhere. Define $\Av_n^\tau(\sigma,\tau)$ similarly.

\begin{theorem}\label{thm:PrGF}
	Let $\sigma\in\S_k$ and $\tau\in\S_\ell$ be incomparable patterns.
	Then \[\Pr(\sigma\prec\tau)=\sum_{n\ge k}\frac{\card{\Av_n^\sigma(\sigma,\tau)}}{n!}.\]
\end{theorem}

\begin{proof}
Let $L_n$ be the event that the first occurrence of either of the patterns $\sigma$ or $\tau$ appears exactly after $n$ draws $X_1,\dots,X_n$. 
Observe that \[\Pr(L_n)=\frac{\card{\Av_n^\sigma(\sigma,\tau)} +\card{\Av_n^\tau(\sigma,\tau)}}{n!},\]
and
\[ \Pr(\sigma\prec\tau\big|L_n)=\frac{\card{\Av_n^\sigma(\sigma,\tau)}}{\card{\Av_n^\sigma(\sigma,\tau)} +\card{\Av_n^\tau(\sigma,\tau)}}. \]
Therefore,
\[
	\Pr(\sigma\prec\tau)=\sum_{n}\Pr(L_n) \Pr(\sigma\prec\tau\big|L_n)
	=\sum_{n}\frac{\card{\Av_n^\sigma(\sigma,\tau)}}{n!}. \qedhere
\]
\end{proof}

For example, since the set $\Av_n^{\iota_k}(\iota_k,21)$ is empty unless $n=k$, in which case it has size $1$, Theorem~\ref{thm:PrGF} gives $\Pr(\iota_k\prec21)=1/k!$, agreeing with equation~\eqref{eq:P1k<21}.

Our next result is a simple consequence of Theorem~\ref{thm:PrGF}. We will use it to find several pairs of patterns $\sigma,\tau$ that have equal probability of appearing first, that is, 
$$\Pr(\sigma\prec\tau)=\Pr(\tau\prec\sigma)=\frac12.$$
We write $\sigma\equiv\tau$ to denote this property, and we say that $\sigma$ and $\tau$ are a {\em tied pair}. For example, Observation~\ref{obs} implies that $\sigma\equiv\comp{\sigma}$ for all $\sigma$.
It is important to note that the relation $\equiv$ is not transitive. 
For example, if $\sigma=132$, $\tau=123$ and $\comp{\tau}=321$, then $\sigma\equiv\tau$, $\tau\equiv\comp{\tau}$, but $\sigma\not\equiv\comp{\tau}$, as we will see in Section~\ref{Sec:Length3}. We noticed that, in fact, there is no pair $\sigma,\tau$ of patterns of length $3$ or $4$ satisfying $\sigma\equiv\tau$ and $\sigma\equiv\comp{\tau}$. An example of non-transitivity that does not involve complements arises in length $4$: if $\sigma=1324$, $\tau=1342$ and $\rho=1432$, we will see that $\sigma\equiv\tau$ and $\tau\equiv\rho$, but $\sigma\not\equiv\rho$.

\begin{corollary}\label{cor:tie}
    If there is a bijection between $\Av_n^{\sigma}(\sigma,\tau)$ and $\Av_n^{\tau}(\sigma,\tau)$ for all $n$, then $\sigma\equiv\tau$.
\end{corollary}

\begin{theorem}\label{thm:iotarho}
For all $2\le i\le k-1$, we have $$\iota_k\equiv12\cdots(i-1)(i+1)\cdots ki.$$
\end{theorem}

\begin{proof}
Let $\sigma=12\cdots(i-1)(i+1)\cdots ki$.
We will construct a bijection between $\Av_n^{\iota_k}(\iota_k,\sigma)$ and $\Av_n^{\sigma}(\iota_k,\sigma)$ for all $n$, and apply Corollary~\ref{cor:tie}. These sets are empty for $n<k$, and the bijection is trivial for $n=k$, so let us assume that $n>k$.

Given $\pi\in\Av_n^{\iota_k}(\iota_k,\sigma)$, its last $k+1$ entries must satisfy
\[ \pi_{n-k}>\pi_{n-k+1}<\pi_{n-k+2}<\dots<\pi_{n},\]
since the only occurrence of $\iota_k$ is at the end of $\pi$. The entry playing the role of $i$ in this occurrence is $\pi_{n-k+i}$. We map $\pi$ to the permutation $\pi'$ obtained by moving this entry to the very end, that is
$$\pi'=\pi_1\pi_2\dots\pi_{n-k+i-1}\pi_{n-k+i+1}\pi_{n-k+i+2}\dots\pi_n\pi_{n-k+i}.$$
See Figure~\ref{fig:iotarho} for an illustration. The resulting permutation belongs to $\Av_n^{\sigma}(\iota_k,\sigma)$, as it ends with an occurrence of $\sigma$, while no other occurrences of $\iota_k$ or $\sigma$ have been created by this move.

\begin{figure}[htb]\centering
\begin{tikzpicture}
	\draw (0,0) -- (1,1);
	\draw[dashed] (1,1) -- (2,2);
	\draw (2,2) -- (2.5,2.5);
	\draw[dashed] (-3,0.5) rectangle (-0.5,3);
	\draw (0,0) -- (-0.5,0.5);
	\draw (8,0) -- (9,1);
	\draw[dashed] (9,1) -- (10,2);
	\draw (10,2) -- (10.5,2.5) -- (11,1.5);
	\draw[dashed] (5,0.5) rectangle (7.5,3);
	\draw (8,0) -- (7.5,0.5);
	\draw [stealth-stealth, thick](3,1) -- (4.5,1);
	\filldraw[black] (0,0) circle (1pt) node[anchor=west]	{1};
	\filldraw[black] (0.5,0.5) circle (1pt) node[anchor=west]	{2};
	\filldraw[black] (1,1) circle (1pt) node[anchor=west]	{3};
	\filldraw[red] (1.5,1.5) circle (1pt) node[anchor=south]{$i$};
	\filldraw[black] (2,2) circle (1pt) node[anchor=west]	{$k-1$};
	\filldraw[black] (2.5,2.5) circle (1pt) node[anchor=west]	{$k$};
	\filldraw[black] (8,0) circle (1pt) node[anchor=west]	{1};
	\filldraw[black] (8.5,0.5) circle (1pt) node[anchor=west]	{2};
	\filldraw[black] (9,1) circle (1pt) node[anchor=west]	{3};
	\filldraw[black] (10,2) circle (1pt) node[anchor=east]	{$k-1$};
	\filldraw[black] (10.5,2.5) circle (1pt) node[anchor=west]	{$k$};
	\filldraw[red] (11,1.5) circle (1pt) node[anchor=south]{$i$};
\end{tikzpicture}
\caption{The bijection in the proof of Theorem~\ref{thm:iotarho}.}
\label{fig:iotarho}
\end{figure}
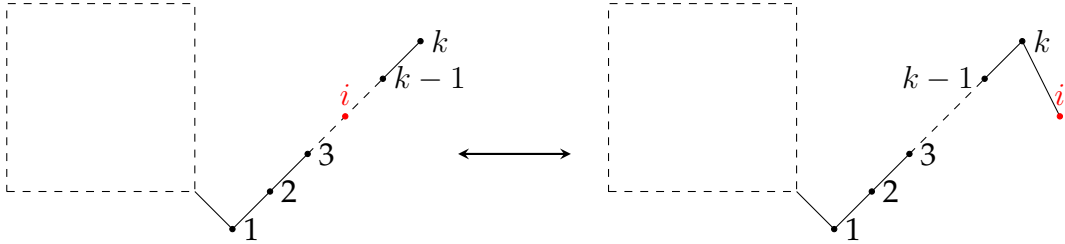

Next we describe the inverse map. Given $\pi'\in\Av_n^{\sigma}(\iota_k,\sigma)$, its last $k+1$ entries must satisfy
\[ \pi'_{n-k}>\pi'_{n-k+1}<\pi'_{n-k+2}<\dots<\pi'_{n-k+i-1}<\pi'_n<\pi'_{n-k+i}<\pi'_{n-k+i+1}<\dots<\pi'_{n-1},\]
using that $\pi'$ avoids $\iota_k$. We map $\pi'$ to the permutation
$$\pi=\pi'_1\pi'_2\dots\pi'_{n-k+i-1}\pi'_n\pi'_{n-k+i}\pi'_{n-k+i+1}\dots\pi'_{n-1},$$ that is, we move entry $\pi'_n$ to the left by inserting it into the preceding increasing run. Note that, since $\pi'_{n-k}>\pi'_{n-k+1}$, this move does not create any occurrence of $\iota_k$ or $\sigma$ in the permutation. It is clear that these two operations are inverses of each other.
\end{proof}

In analogy to the case of words, for any $\sigma,\tau\in\S_k$, we define the set
$$\cB_{\sigma,\tau}=\{i\in[k-1]:\st(\sigma_{k-i+1}\dots\sigma_k)=\st(\tau_1\dots\tau_{i})\}.$$
Note that $1$ always belongs to this set. If neither $\cB_{\sigma,\tau}$ nor $\cB_{\tau,\sigma}$ contain any larger elements, we say that $\sigma$ and $\tau$ are {\em non-overlapping}, following~\cite{ElizaldeNoy5,DwyerSergi}.
More generally, for any $j\ge2$, we say that $\sigma$ and $\tau$ are {\em non-$j$-overlapping} if
neither $\cB_{\sigma,\tau}$ nor $\cB_{\tau,\sigma}$  contain any elements greater than or equal to $j$, that is, if $\cB_{\sigma,\tau}\cup\cB_{\tau,\sigma}\subseteq[j-1]$. In particular, non-$2$-overlapping means the same as non-overlapping.

We say that $\sigma$ is {\em non-$j$-self-overlapping} if $\cB_{\sigma,\sigma}\subseteq[j-1]$.
For $j=2$, we will simply say that $\sigma$ is non-self-overlapping.

\begin{theorem}\label{thm:NoOverlap}
    Let $\sigma,\tau\in\S_k$, and suppose that there exists $j\in[k-1]$ such that
    \begin{enumerate}[label=(\roman*)]
        \item  $\sigma_i=\tau_i$ for $1\le i\le j$,
        \item $\sigma$ and $\tau$ are non-$(j+1)$-overlapping, and
        \item each of $\sigma$ and $\tau$ is non-$(j+1)$-self-overlapping. 
    \end{enumerate}
    Then $\sigma\equiv\tau$.
\end{theorem}

\begin{proof}
By Corollary~\ref{cor:tie}, it suffices to describe a bijection between $\Av_n^\sigma(\sigma,\tau)$ and $\Av_n^\tau(\sigma,\tau)$ for all $n$.

    Given $\pi\in\Av_n^\sigma(\sigma,\tau)$, its last $k$ entries must satisfy $\st(\pi_{n-k+1}\dots\pi_n)=\sigma$.
    We map $\pi$ to the permutation $\omega$ obtained by rearranging the last $k$ entries of $\pi$ to be in the same relative order as $\tau$, that is, $\omega\in\S_n$ is the only permutation such that $\omega_i=\pi_i$ for $i\in[n-k]$, and  \begin{equation}\label{eq:omegatau}\st(\omega_{n-k+1}\dots\omega_n)=\tau.\end{equation}
    Note that condition (i) implies that $\omega_i=\pi_i$ for $n-k+1\le i\le n-k+j$. 
    
    Clearly, $\omega$ ends with $\tau$. Thus, to show that $\omega\in\Av_n^\tau(\sigma,\tau)$, it remains to prove that $\omega$ avoids $\sigma$ and $\tau$ elsewhere.
    Suppose, for contradiction, that $\omega$ has an occurrence of $\sigma$. Since the common prefix $\omega_1\dots\omega_{n-k+j}=\pi_1\dots\pi_{n-k+j}$ avoids $\sigma$, such an occurrence must overlap with $\omega_{n-k+1}\dots\omega_n$ in at least $j+1$ positions. But, using equation~\eqref{eq:omegatau}, this contradicts condition (ii).
    Similarly, if $\omega$ had an occurrence of $\tau$ other than the one at the end, it would contradict condition (iii) for $\tau$.
    
   The inverse map can be described similarly: given $\omega\in\Av_n^\tau(\sigma,\tau)$, we rearrange its last $k$ entries to put them in the same relative order as $\sigma$.
\end{proof}

\begin{example}
The patterns $24513$ and $24531$ agree in the first $3$ positions, they are non-$4$-overlapping, and each of them is non-$4$-self-overlapping. Thus, by Theorem~\ref{thm:NoOverlap}, we have $24513\equiv24531$.
On the other hand, we will see in Section~\ref{Sec:Length4} that
$2314\not\equiv2341$. These patterns agree in the first $2$ positions, but they fail the non-$3$-overlapping condition, since $\cB_{2341,2314}=\{1,3\}$.
\end{example}

The case $j=1$ of Theorem~\ref{thm:NoOverlap} is already quite powerful. 
In this case, $\sigma$ and $\tau$ are non-overlapping and non-self-overlapping, and they start with the same letter. An example of this situation is the tie $2341\equiv 2431$. Another example is the family of patterns in the next corollary. Notice that this result cannot be deduced directly from Theorem~\ref{thm:iotarho}, since $\equiv$ is not a transitive relation.

\begin{corollary}
    \label{cor:rhorho'}
	For all $2\le i<i'\le k-1$, we have $$12\cdots (i-1)(i+1)\cdots ki\equiv 12\cdots(i'-1)(i'+1)\cdots ki'.$$ 
\end{corollary}

Next we give a family of tied pairs of patterns that can overlap in $2$ positions, but not in $3$ or more positions. This results follows from Theorem~\ref{thm:NoOverlap} with $j=2$.

\begin{corollary} For all $k\ge5$, we have 
\[134\dots(k-2)k2(k-1)\equiv134\dots(k-2)(k-1)2k.\]
\end{corollary}

We conclude this section with another family of tied pairs of overlapping patterns, which cannot be deduced from Theorems~\ref{thm:iotarho} and~\ref{thm:NoOverlap}. 

\begin{theorem}\label{thm:1k2k-1general} 
Let $k\ge4$, and let $\sigma=1k\alpha(k-2)(k-1)$ and $\tau=1(k-1)\beta(k-2)k$, where $\alpha$ and $\beta$ are permutations of $\{2,3,\dots,k-3\}$.
Then $\sigma\equiv\tau$.
\end{theorem}

\begin{figure}[htb]\centering
\begin{tikzpicture}
	\filldraw[black] (0,0) circle (2pt) node[left]	{1};
	\filldraw[black] (0.5,3.5) circle (2pt) node[left]	{$k$};
	\draw[dashed] (.75,.25) rectangle (2.75,2.25);
	\draw (1.75,1.25) node {$\alpha$};
	\filldraw[black] (3,2.5) circle (2pt) node[right]	{$k-2$};
	\filldraw[black] (3.5,3) circle (2pt) node[right]	{$k-1$};
	\draw (5.75,1.5) node {$\equiv$};
	\begin{scope}[shift={(8,0)}]
		\filldraw[black] (0,0) circle (2pt) node[left]	{1};
	\filldraw[black] (0.5,3) circle (2pt) node[left]	{$k-1$};
	\draw[dashed] (.75,.25) rectangle (2.75,2.25);
	\draw (1.75,1.25) node {$\beta$};
	\filldraw[black] (3,2.5) circle (2pt) node[right]	{$k-2$};
	\filldraw[black] (3.5,3.5) circle (2pt) node[right]	{$k$};
	\end{scope}
\end{tikzpicture}
\caption{The patterns $\sigma$ and $\tau$ in Theorem~\ref{thm:1k2k-1general}.}
\label{fig:1k2k-1general}
\end{figure}

\begin{proof}
As before, it suffices to construct a bijection between $\Av_n^\sigma(\sigma,\tau)$ and $\Av_n^\tau(\sigma,\tau)$ for all~$n$.
Given $\pi\in\Av_n^\sigma(\sigma,\tau)$, its last $k$ entries must satisfy
\begin{equation}\label{eqn:pi1k2k-1}
        \st(\pi_{n-k+1}\pi_{n-k+2}\dots\pi_n)=\sigma.
\end{equation}
We map $\pi$ to the permutation $\omega$ obtained by reordering these last $k$ entries so that they form an occurrence of $\tau$, that is,
$$\st(\omega_{n-k+1}\omega_{n-k+2}\dots\omega_n)=\tau,$$
while keeping $\omega_i=\pi_i$ for all $i\in[n-k]$. Note that $\omega_{n-k+1}=\pi_{n-k+1}$, $\omega_{n-1}=\pi_{n-1}$, $\omega_{n-k+2}=\pi_n$, and $\omega_n=\pi_{n-k+2}$. 

By construction, $\omega$ ends with an occurrence of $\tau$. To show that $\omega\in\Av_n^\tau(\sigma,\tau)$, it remains to show that it has no other occurrence of $\sigma$ or $\tau$.

Since the common prefix $\omega_1\dots\omega_{n-k+1}=\pi_1\dots\pi_{n-k+1}$ avoids the patterns $\sigma$ and $\tau$, any occurrence of such patterns must overlap with $\omega_{n-k+1}\dots\omega_n$ in at least two positions. In fact, the overlap must be in exactly two positions, noting that $\cB_{\sigma,\tau}=\cB_{\tau,\tau}=\{1,2\}$. So, the only possible occurrence of $\sigma$ and $\tau$ (other than the one at the end) in $\omega$ must be $\omega_{n-2k+3}\omega_{n-2k+4}\dots\omega_{n-k+2}$.

Consider first the case that $\st(\omega_{n-2k+3}\omega_{n-2k+4}\dots\omega_{n-k+2})=\sigma$, or equivalently, 
\begin{equation}\label{eqn:gamma1k2k-1} 
    \st(\pi_{n-2k+3}\pi_{n-2k+4}\dots\pi_{n-k+1}\pi_{n})=\sigma.
\end{equation}
By equation~\eqref{eqn:pi1k2k-1}, $\pi_n<\pi_{n-k+2}$. We now have two cases depending on the relationship between $\pi_{n-k+2}$ and $\pi_{n-2k+4}$. If $\pi_{n-k+2}<\pi_{n-2k+4}$, equation~\eqref{eqn:gamma1k2k-1} implies that
$$\st(\pi_{n-2k+3}\pi_{n-2k+4}\dots\pi_{n-k+1}\pi_{n-k+2})=\sigma,$$ contradicting the fact that $\pi\in\Av_n^\sigma(\sigma,\tau)$.
If $\pi_{n-2k+4}<\pi_{n-k+2}$, equation~\eqref{eqn:gamma1k2k-1} implies that 
$$\st(\pi_{n-2k+3}\pi_{n-2k+4}\dots\pi_{n-k+1}\pi_{n-k+2})=\tau,$$ 
again a contradiction.

Consider now the case that $\st(\omega_{n-2k+3}\omega_{n-2k+4}\dots\omega_{n-k+2})=\tau$.
Since $\omega_{n-k+2}=\pi_n<\pi_{n-k+2}$, this implies that
$$\st(\pi_{n-2k+3}\pi_{n-2k+4}\dots\pi_{n-k+1}\pi_{n-k+2})=\tau,$$ the same contradiction as before.

The inverse map has a very similar description: given $\omega\in\Av_n^\tau(\sigma,\tau)$, we let $\pi$ be the permutation obtained by switching the entries $\pi_{n-k+2}$ and $\pi_n$. It is clear that $\pi$ ends with an occurrence of $\sigma$. The fact that $\pi$ avoids $\sigma$ and $\tau$ elsewhere can be prove using an argument analogous to the one given above, noting now that $\cB_{\tau,\sigma}=\cB_{\sigma,\sigma}=\{1,2\}$.
\end{proof}

For example, the above theorem implies that $1734256\equiv1623457$. Figure~\ref{fig:1k2k-1general} shows the general shape of the patterns in this theorem. 

It is an interesting open problem to characterize all pairs of patterns $\sigma,\tau$ for which $\sigma\equiv\tau$.


\section{Winning probabilities for patterns of length 3}
\label{Sec:Length3}

In this section we focus on patterns of length~3. We give formulas for the probabilities $\Pr(\sigma\prec\tau)$ for all pairs $\sigma,\tau\in\S_3$. A summary of the numerical values is given in Table~\ref{tab:length3}. It is enough to compute the entries that are highlighted in \textcolor{orange}{orange}, since the others follow from Observation~\ref{obs}. In the rest of this section, we compute each of these six values.

\begin{table}[tbh]
    \centering
\begin{tabular}{c||c|c|c|c|c|c}
\diagbox{$\sigma$}{$\tau$} & 123 & 132 & 213 & 231 & 312 & 321 \\ \hline\hline
	123 & -- & \orange{0.5} & \orange{0.412} & \orange{0.550} & \orange{0.342} & 0.5\\ \hline
	132 & 0.5 & -- & \orange{0.461} & \orange{0.476} & 0.5 & 0.658 \\ \hline
	213 & 0.588 & 0.539 & -- & 0.5 & 0.524 & 0.450 \\ \hline
	231 & 0.450 & 0.524 & 0.5 & -- & 0.539 & 0.588 \\ \hline
	312 & 0.658 & 0.5 & 0.476 & 0.461 & -- & 0.5 \\ \hline
	321 & 0.5 & 0.342 & 0.550 & 0.412 & 0.5 & -- \\
\end{tabular}
\caption{The probabilities $\Pr(\sigma\prec\tau)$ for all pairs of patterns of length $3$. 
Except for the values $0.5$, which are exact, only the first 3 significant digits are given.}
\label{tab:length3}
\end{table}


\subsection{$123$ vs.\ $132$}

Theorem~\ref{thm:iotarho} with $k=3$ and $i=2$ implies that $123\equiv132$. All the other tied pairs of patterns of length $3$ follow now from Observation~\ref{obs}, and are depicted in Figure~\ref{fig:TieLength3}.

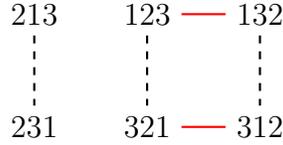
\begin{figure}[htb]
    \centering
    \begin{tikzpicture}[scale=.75]
        \node (213) at (0,1) {$213$};
        \node (231) at (0,-1) {$231$};
        \node (123) at (2,1) {$123$};
        \node (132) at (4,1) {$132$};
        \node (321) at (2,-1) {$321$};
        \node (312) at (4,-1) {$312$};
        \draw [dashed,thick] (213)--(231);
        \draw [dashed,thick] (123)--(321);
        \draw [dashed,thick] (132)--(312);
        \draw [red,thick] (123)--(132);
        \draw [red,thick] (321)--(312);
    \end{tikzpicture}
    \caption{All tied pairs of patterns in $\S_3$. Dashed edges indicate pairs of the form $\sigma\equiv\comp{\sigma}$.}
    \label{fig:TieLength3}
\end{figure}


\subsection{$132$ vs.\ $231$}

\begin{theorem}\label{thm:132vs231}
	\[\Pr(132\prec231)=\frac{e^2-2e-1}{2}\approx 0.476.\]
\end{theorem}

\begin{proof}
We will compute the numbers $\card{\Av_n^{132}(132,231)}$ and apply Theorem~\ref{thm:PrGF}.
A permutation avoids $132$ and $231$ if and only if it has no peaks, 
and so $\pi\in\Av_n^{132}(132,231)$ if and only if there exists some $i\in[n-2]$ such that
\[ \pi_1>\pi_2>\dots>\pi_i<\pi_{i+1}<\dots<\pi_{n-2}<\pi_{n-1}>\pi_{n}\]
where $\pi_{n-2}<\pi_{n}$. Switching the last two entries, this condition is equivalent to 
\[ \pi_1>\pi_2>\dots>\pi_i<\pi_{i+1}<\dots<\pi_{n-2}<\pi_{n}<\pi_{n-1}.\]
For given $i$, a permutation satisfying this condition is determined by the choice of the entries to the left of $\pi_i$. It follows that 
\[\card{\Av_n^{132}(132,231)}=\sum_{i=1}^{n-2}{\binom{n-1}{i-1}}=2^{n-1}-n.\] Hence, by Theorem~\ref{thm:PrGF}, the desired probability is 
\[
	\Pr(132\prec231)=\sum_{n\ge3}\frac{2^{n-1}-n}{n!}=\frac12\sum_{n\ge3}\frac{2^{n}}{n!}-\sum_{n\ge2}\frac{1}{n!}=\frac{e^2-2e-1}{2}. \qedhere
\]
\end{proof}

We remark that the expression for $\Pr(\sigma\prec\tau)$ given in Theorem~\ref{thm:132vs231} is not a rational number, in contrast with Conway's formula for the case of words (Theorem~\ref{thm:Conway}). This suggests that we cannot expect to have a similar simple formula for the case of permutations.


\subsection{$123$ vs.\ $213$}

In order to determine $\card{\Av_n^{123}(123,213)}$, it will be helpful to compute $\card{\Av_n^{312}(123,213)}$ first.

\begin{lemma}\label{lem:bn}
	Let $b_n=\card{\Av_n^{312}(123,213)}$. Then for all $n\ge5$, \[b_n=b_{n-1}+(n-1) b_{n-2},\] with initial conditions $b_0=b_1=b_2=0$, $b_3=1$ and $b_4=4$.
The corresponding exponential generating function is
\[ B(x)=e^{x+\frac{x^2}{2}}\left(1-\int_0^x e^{-t-\frac{t^2}{2}}dt\right)-1-\frac{x^2}{2}.\]
\end{lemma}

\begin{proof}
In any permutation that avoids $123$ and $213$, the entry $n$ must occur in one of the first two positions. Thus, for $n\ge5$, any $\pi\in\Av_n^{312}(123,213)$ is of the form $\pi=n\pi'$, where 
$\pi'\in\Av_{n-1}^{312}(123,213)$, or of the form $\pi=in\pi'$, where $i\in[n-1]$ and $\st(\pi')$ belongs to $\Av_{n-2}^{312}(123,213)$.
Since we have $n-1$ choices for $i$, this decomposition gives the stated recurrence.
The initial conditions are obtained easily noting that any $\pi\in\Av_n^{312}(123,213)$ must end with an occurrence of $312$.

	From the recurrence, we obtain
 \[	\sum_{n\ge5}b_n\frac{x^{n-1}}{(n-1)!}=\sum_{n\ge5}b_{n-1}\frac{x^{n-1}}{(n-1)!}+\sum_{n\ge5}b_{n-2}\frac{x^{n-1}}{(n-2)!},\]
 which is equivalent to
\[ B'(x)-\frac{x^2}{2}-4\frac{x^3}{3!}=B(x)-\frac{x^3}{3!}+xB(x),\]
 \[ B'(x)=(1+x)B(x)+\frac{x^2+x^3}{2}.\]
 Solving this differential equation and using that $B(0)=0$, we obtain the stated formula for $B(x)$.
\end{proof}

The same sequence, up to initial terms, appears in work of Kitaev and Mansour~\cite[Theorem 24(ii)]{KitaevMansour} (see also \cite[A059480]{OEIS}), where it has an interpretation in terms of vincular patterns. Specifically, they consider permutations that avoid the patterns 
$12\dash3$ and $21\dash3$, and end with a $12$. Let us show that these coincide with our permutations.

\begin{proposition} For $n\ge3$, we have $\Av_n^{312}(123,213)=\Av_n^{12}(12\dash3,21\dash3)$.
\end{proposition}

\begin{proof}
If $n\ge3$ and $\pi\in\Av_n^{12}(12\dash3,21\dash3)$, the last three entries of $\pi$ must form an occurrence of $312$, since otherwise they would create 
an occurrence of $12\dash3$ or $21\dash3$.
Additionally, avoiding $12\dash3$ or $21\dash3$ requires avoiding $123$ or $213$ by definition, so $\pi\in\Av_n^{312}(123,213)$.

Conversely, given $\pi\in\Av_n^{312}(123,213)$, we want to show that it avoids $12\dash3$ and $21\dash3$. Suppose, for the sake of contradiction, that it contains occurrences of these patterns, and pick an occurrence $\pi_i\pi_{i+1}\pi_{j}$ of minimum width, defined as $j-i+1\ge3$. If $j=i+2$ or $\pi_{i+2}\ge\pi_j$, then $\pi_i\pi_{i+1}\pi_{i+2}$ would be an occurrence of $123$ or a $213$, so we must have $j>i+2$ and $\pi_{i+2}<\pi_j$. But then $\pi_{i+1}\pi_{i+2}\pi_{j}$ would be an occurrence of $12\dash3$ and $21\dash3$ with smaller width than the previously chosen one, contradicting its minimality.
\end{proof}

In the following result, $b_n$ and $B(x)$ are the ones given by Lemma~\ref{lem:bn}.

\begin{theorem}\label{thm:cn}
	Let $a_n=\card{\Av_n^{123}(123,213)}$. Then for all $n\ge5$, \[a_n=a_{n-1}+(n-1) a_{n-2}+b_{n-1},\] with initial conditions $a_0=a_1=a_2=0$, $a_3=1$ and $a_4=2$.
	The corresponding exponential generating function is 
 \[A(x) = e^{x+\frac{x^2}{2}}\left((1+x)-(2+x)\int_0^x e^{-t-\frac{t^2}{2}}dt\right)-1.\]
\end{theorem}

\begin{proof}
In any $\pi\in\Av_n^{123}(123,213)$ with $n\ge5$, the entry $n$ must appear in the first, the second, or the last position. If it appears in the first or the second position, we argue as in the proof of Lemma~\ref{lem:bn} to obtain the terms $a_{n-1}+(n-1)a_{n-2}$. If it appears in the last position, then we can write $\pi=\pi'n$, where $\pi'\in \Av_{n-1}^{312}(123,213)$; indeed, since $\pi_{n-2}<\pi_{n-1}$, we must have $\pi_{n-3}>\pi_{n-1}$, because otherwise $\pi_{n-3}\pi_{n-2}\pi_{n-1}$ would form an occurrence of $123$ or $213$.

From the recurrence, we obtain \[
	\sum_{n\ge5}a_n\frac{x^{n-1}}{(n-1)!}=\sum_{n\ge5}a_{n-1}\frac{x^{n-1}}{(n-1)!}+\sum_{n\ge5}a_{n-2}\frac{x^{n-1}}{(n-2)!}+\sum_{n\ge5}b_{n-1}\frac{x^{n-1}}{(n-1)!},\]
which is equivalent to
\[ A'(x)-\frac{x^2}{2}-2\frac{x^3}{3!}=A(x)-\frac{x^3}{3!}+xA(x)+B(x)-\frac{x^3}{3!},\]
\[A'(x)=(1+x)A(x)+B(x)+\frac{x^2}{2}.\]
Solving this differential equation and using that $A(0)=0$, we get 
\[A(x)=e^{x+\frac{x^2}{2}}\left(\int_0^x(B(t)+\frac{t^2}{2})e^{-t-\frac{t^2}{2}}dt\right),\]
which, using Lemma~\ref{lem:bn}, equals the stated expression.
\end{proof}

The desired probability can now be computed from Theorems~\ref{thm:PrGF} and~\ref{thm:cn}.

\begin{corollary}    
\[\Pr(123\prec213)=A(1)=e^{\frac32}\left(2-3\int_0^1e^{-t-\frac{t^2}{2}}dt\right)-1 \approx0.412.\]
\end{corollary} 


\subsection{$123$ vs.\ $231$}

In the remaining three cases we do not have closed formulas for the probability.
The following result is proved similarly to \cite[Proposition 1]{KitaevMansour}, where
Kitaev and Mansour give a recurrence for the number of permutations avoiding $123$ and $231$, but without the condition of ending with an occurrence of $123$.

\begin{theorem}
	Let $d(n,i)=|\{\pi\in\Av_n^{123}(123,231):\pi_1=i\}|$. Then for all $n\ge5$ and all $i\in[n]$, \[d(n,i)=\sum_{j=1}^{i-1}d(n-1,j)+\sum_{j=i}^{n-2}(n-1-j)d(n-2,j),\] with $d(3,1)=1$, $d(3,2)=d(3,3)=0$; and $d(4,1)=0$, $d(4,2)=d(4,3)=d(4,4)=1$.
\end{theorem}

\begin{proof}
To count permutations $\pi=\pi_1\pi_2\dots\pi_n\in\Av_n^{123}(123,231)$ with $\pi_1=i$, we consider two cases.

If $\pi_1>\pi_2$, then $\pi_1\pi_2\pi_3$ cannot form a $123$ or a $231$. Thus, $\st(\pi_2\dots\pi_n)$ can be any permutation in $\Av_{n-1}^{123}(123,231)$ that starts with an entry $j<i$. This gives the first term of the recurrence.

If $\pi_1<\pi_2$, in order for $\pi$ to avoid $123$ and $231$, $\pi_1\pi_2\pi_3$ must form a pattern $132$. In this case, $\st(\pi_3\dots\pi_n)$ can be any permutation in $\Av_{n-2}^{123}(123,231)$ that starts with an entry $j\ge i$. Given such a permutation, after inserting $\pi_1=i$ and shifting the values greater than or equal to $i$ up by one, 
the entry $\pi_2$ can take any of the values $\{j+2,j+3,\dots,n\}$, leaving $n-j-1$ choices.
\end{proof}

\begin{corollary}\label{cor:123231}
 \[\Pr(123\prec231)=\sum_{n\ge3}\frac{1}{n!}\sum_{i=1}^n d(n,i)\approx0.550.\]
\end{corollary}

An interesting consequence of Corollary~\ref{cor:123231} is that, among $123$ and $231$, the pattern $123$ is more likely to appear first in the random sequence $\X$. However, counterintuitively, the expected time until the first occurrence of this pattern is actually larger, that is,
$\ET_{123}>\ET_{231}$, by equation~\eqref{eq:ET123} and Corollary~\ref{cor:ET}.


\subsection{$132$ vs.\ $213$}

For this case and the next we obtain three-parameter recurrences, which are proved similarly.

\begin{theorem}
	Let $s(n; i, j)=|\{\pi\in\Av_n^{132}(132,213):\pi_1=i,\pi_2=j\}|$. Then, for all $n\ge4$ and $i,j\in[n]$ with $i\neq j$, 
 \[
		s(n; i, j)=
  \begin{cases}
		\ds\sum_{\ell=1}^{j-1}s(n-1; j, \ell)+\sum_{\ell=j+1}^{i-1}s(n-1; j, \ell) & \text{if $i>j$},\\
		\ds\sum_{\ell=1}^{i-1}s(n-1; j-1, \ell)+\sum_{\ell=j}^{n-1}s(n-1; j-1, \ell) & \text{ if $i<j$},
  \end{cases}
	\] with initial conditions $s(3; 1,3)=1$, and $s(3; i, j)=0$ for all other $i,j$.
\end{theorem}

\begin{proof}
	Let $\pi=\pi_1\pi_2\cdots\pi_n\in\Av_n^{132}(132,213)$ with $\pi_1=i$ and $\pi_2=j$. If $i>j$, in order for $\pi$ to avoid $213$, we must have $\pi_3<i$. In this case, $\st(\pi_2\dots\pi_n)$ can be any permutation in $\Av_{n-1}^{132}(132,213)$ that starts with $j$ followed by an entry $\ell<i$.
 
	If $i<j$, in order for $\pi$ to avoid $132$, we must have $\pi_3<i$ or $\pi_3>j$. In this case, 
    $\st(\pi_2\dots\pi_n)$ can be any permutation in $\Av_{n-1}^{132}(132,213)$ that starts with $j-1$ (this entry becomes $\pi_2=j$ after $\pi_1=i$ is inserted) followed by an entry $\ell$ such that $\ell<i$ or $\ell\ge j$.
\end{proof}

\begin{corollary}
     \[\Pr(132\prec213)=\sum_{n\ge3}\sum_{\substack{i,j=1\\ i\neq j}}^n\frac{s(n;i,j)}{n!}\approx0.461.\]
\end{corollary}


\subsection{$123$ vs.\ $312$}

\begin{theorem}
	Let $t(n; i, j)=|\{\pi\in\Av_n^{123}(123,312):\pi_1=i,\pi_2=j\}|$. Then, for all $n\ge4$ and $i,j\in[n]$ with $i\neq j$,
\[		t(n; i, j)=
  \begin{cases}
		\ds\sum_{\ell=1}^{j-1}t(n-1; j, \ell)+\sum_{\ell=i}^{n-1}t(n-1; j, \ell) & \text{if $i>j$},\\
		\ds\sum_{\ell=1}^{j-2}t(n-1; j-1, \ell) & \text{ if $i<j$},
  \end{cases}
  \]
	with initial conditions $t(3;1,2)=1$, and $t(3;i,j)=0$ for all other $i,j$.
 \end{theorem}

\begin{proof}
	Let $\pi=\pi_1\pi_2\cdots\pi_n\in\Av_n^{123}(123,312)$ with $\pi_1=i$ and $\pi_2=j$. If $i>j$, in order for $\pi$ to avoid $312$, we must have $\pi_3<j$ or $\pi_3>i$. In this case, $\st(\pi_2\dots\pi_n)$ can be any permutation in $\Av_{n-1}^{123}(123,312)$ that starts with $j$ followed by an entry  $\ell$ such that $\ell<j$ or $\ell\ge i$.
 
	If $i<j$, in order for $\pi$ to avoid $123$, we must have $\pi_3<j$. In this case, $\st(\pi_2\dots\pi_n)$ can be any permutation in $\Av_{n-1}^{123}(123,312)$ that starts with $j-1$ (this entry becomes $\pi_2=j$ after $\pi_1=i$ is inserted) followed by an entry $\ell<j-1$.
\end{proof}

\begin{corollary}
 \[\Pr(123\prec312)=\sum_{n\ge3}\sum_{\substack{i,j=1\\ i\neq j}}^n\frac{t(n;i,j)}{n!}\approx0.342.\]
\end{corollary}


\section{Some patterns of length 4}
\label{Sec:Length4}

In this section we consider patterns of length 4. Even though we will not give formulas for $\Pr(\sigma\prec\tau)$ for all pairs $\sigma,\tau\in\S_4$, we will describe all the tied pairs, i.e., pairs where this probability is $1/2$.
Table~\ref{tab:length4} summarizes some of the numerical data, computed using
Theorem~\ref{thm:PrGF}, with the series truncated at $n=11$. The table describes, for each pair of patterns of length $4$, which pattern is more likely to appear first in a random sequence.

\begin{table}[tbh]
    \centering
    \scalebox{0.8}{\begin{tabular}{c||c|c|c|c|c|c|c|c|c|c|c|c|c|c|c|c|c|c|c|c|c|c|c|c|}
         \backslashbox{\cya{$\sigma$}}{\textcolor{magenta}{$\tau$}} & \rot{1234} & \rot{1243} & \rot{1324} & \rot{1342} & \rot{1423} & \rot{1432} & \rot{2134} & \rot{2143} & \rot{2314} & \rot{2341} & \rot{2413} & \rot{2431} & \rot{3124} & \rot{3142} & \rot{3214} & \rot{3241} & \rot{3412} & \rot{3421} & \rot{4123} & \rot{4132} & \rot{4213} & \rot{4231} & \rot{4312} & \rot{4321} \\\hline\hline
         \cya{1234} & &  \tie & \loss & \tie & \loss & \loss & \loss & \loss & \loss & \win  & \loss & \loss & \loss & \loss & \loss & \loss & \loss & \loss & \loss & \loss & \loss & \loss & \loss &  \tie\\ \hline
         \cya{1243} &  \tie & & \win &  \tie & \win & \win & \loss & \win & \win & \loss & \win & \win & \loss & \win & \win & \win & \loss & \loss & \loss & \win & \win & \win &  \tie & \win \\  \hline
         \cya{1324} & \win & \loss & &  \tie &  \tie & \win & \win & \win & \loss & \loss & \win & \win & \loss & \win & \loss & \win & \win & \loss & \loss & \loss & \loss &  \tie & \loss & \win \\ \hline
         \cya{1342} &  \tie &  \tie &  \tie & & \loss &  \tie & \loss & \win & \win & \loss & \win & \loss & \loss & \win & \win & \win & \win & \win & \loss & \win &  \tie & \win & \loss & \win \\ \hline
         \cya{1423} & \win & \loss &  \tie & \win & &  \tie & \loss & \loss & \loss & \win & \win & \win & \win & \loss & \loss & \loss & \win & \win & \win &  \tie & \loss & \win & \loss & \win \\ \hline
         \cya{1432} & \win & \loss & \loss &  \tie &  \tie & & \win & \loss & \loss & \loss & \win & \loss & \win & \loss & \win & \win & \loss & \loss &  \tie & \loss & \win & \win & \win & \win \\ \hline
         \cya{2134} & \win & \win & \loss & \win & \win & \loss & &  \tie & \loss & \win & \win & \loss & \win & \win & \loss &  \tie & \win &  \tie & \win & \loss & \loss & \win & \win & \win \\ \hline
         \cya{2143} & \win & \loss & \loss & \loss & \win & \win &  \tie & & \loss & \loss & \win & \win & \win & \win & \loss & \win &  \tie & \loss & \win & \loss & \loss & \loss & \win & \win \\ \hline
         \cya{2314} & \win & \loss & \win & \loss & \win & \win & \win & \win & & \loss & \win & \win & \loss & \win & \loss &  \tie & \loss &  \tie & \loss & \win & \loss & \loss & \loss & \win \\ \hline
         \cya{2341} & \loss & \win & \win & \win & \loss & \win & \loss & \win & \win & & \win &  \tie & \loss & \win &  \tie & \win & \win & \win & \loss & \win & \loss & \win & \loss & \win \\ \hline
         \cya{2413} & \win & \loss & \loss & \loss & \loss & \loss & \loss & \loss & \loss & \loss & &  \tie & \win &  \tie & \loss & \loss & \loss & \loss & \win & \win & \loss & \loss & \loss & \win \\ \hline
         \cya{2431} & \win & \loss & \loss & \win & \loss & \win & \win & \loss & \loss &  \tie &  \tie & &  \tie & \loss & \win & \win & \loss & \loss & \loss & \loss & \win & \win & \win & \win \\ \hline
         \cya{3124} & \win & \win & \win & \win & \loss & \loss & \loss & \loss & \win & \win & \loss &  \tie & &  \tie &  \tie & \loss & \loss & \win & \win & \loss & \win & \loss & \loss & \win \\ \hline
         \cya{3142} & \win & \loss & \loss & \loss & \win & \win & \loss & \loss & \loss & \loss &  \tie & \win &  \tie & & \loss & \loss & \loss & \loss & \loss & \loss & \loss & \loss & \loss & \win \\ \hline
         \cya{3214} & \win & \loss & \win & \loss & \win & \loss & \win & \win & \win &  \tie & \win & \loss &  \tie & \win & & \win & \win & \loss & \win & \loss & \win & \win & \win & \loss \\ \hline
         \cya{3241} & \win & \loss & \loss & \loss & \win & \loss &  \tie & \loss &  \tie & \loss & \win & \loss & \win & \win & \loss & & \win & \win & \win & \win & \loss & \win & \loss & \win \\ \hline
         \cya{3412} & \win & \win & \loss & \loss & \loss & \win & \loss &  \tie & \win & \loss & \win & \win & \win & \win & \loss & \loss & &  \tie & \win & \win & \loss & \loss & \loss & \win \\ \hline
         \cya{3421} & \win & \win & \win & \loss & \loss & \win &  \tie & \win &  \tie & \loss & \win & \win & \loss & \win & \win & \loss &  \tie & & \loss & \win & \win & \loss & \win & \win \\ \hline
         \cya{4123} & \win & \win & \win & \win & \loss &  \tie & \loss & \loss & \win & \win & \loss & \win & \loss & \win & \loss & \loss & \loss & \win & &  \tie &  \tie & \loss & \loss & \win \\ \hline
         \cya{4132} & \win & \loss & \win & \loss &  \tie & \win & \win & \win & \loss & \loss & \loss & \win & \win & \win & \win & \loss & \loss & \loss &  \tie & & \win &  \tie & \loss & \win \\ \hline
         \cya{4213} & \win & \loss & \win &  \tie & \win & \loss & \win & \win & \win & \win & \win & \loss & \loss & \win & \loss & \win & \win & \loss &  \tie & \loss & &  \tie &  \tie &  \tie \\ \hline
         \cya{4231} & \win & \loss &  \tie & \loss & \loss & \loss & \loss & \win & \win & \loss & \win & \loss & \win & \win & \loss & \loss & \win & \win & \win &  \tie &  \tie & & \loss & \win \\ \hline
         \cya{4312} & \win &  \tie & \win & \win & \win & \loss & \loss & \loss & \win & \win & \win & \loss & \win & \win & \loss & \win & \win & \loss & \win & \win &  \tie & \win & &  \tie \\ \hline
         \cya{4321} &  \tie & \loss & \loss & \loss & \loss & \loss & \loss & \loss & \loss & \loss & \loss & \loss & \loss & \loss & \win & \loss & \loss & \loss & \loss & \loss &  \tie & \loss &  \tie & \\ \hline
    \end{tabular}}
    \caption{The probabilities $\Pr(\sigma\prec\tau)$ for all pairs of length $4$. Cells in \textcolor{blue}{blue} represent ties, cells in \textcolor{magenta}{magenta} represent values $<1/2$ (i.e., \textcolor{magenta}{$\tau$} is more likely to appear first), and cells in \textcolor{cyan}{cyan} represent values $>1/2$ (i.e., \cya{$\sigma$} is more likely to appear first).}
    \label{tab:length4}
\end{table}

The numerical data suggests that, in addition to the twelve pairs of the form $\sigma\equiv\comp{\sigma}$, there are $11$ non-trivial tied pairs $\sigma\equiv\tau$, along with the corresponding pairs $\comp{\sigma}\equiv\comp{\tau}$ that follow from these by Observation~\ref{obs}. All the tied pairs of patterns of length~4 are depicted in Figure~\ref{fig:TieLength4}.

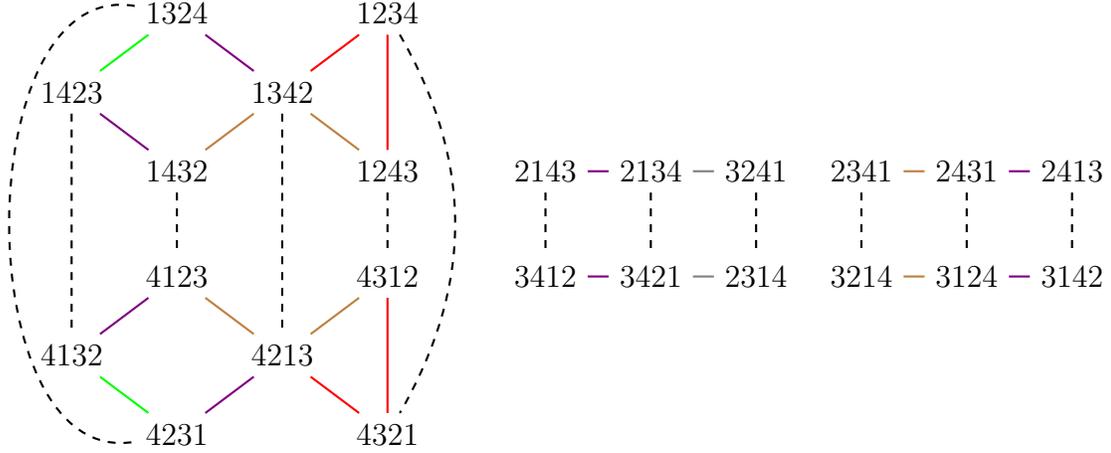
\begin{figure}[htb]
    \centering
    \begin{tikzpicture}[scale=.685]
        \node (1423) at (0,5) {$1423$};
        \node (1324) at (2,6.5) {$1324$};
        \node (1432) at (2,3.5) {$1432$};
        \node (1342) at (4,5) {$1342$};
        \node (1234) at (6,6.5) {$1234$};
        \node (1243) at (6,3.5) {$1243$};
        \node (2143) at (9,3.5) {$2143$};
        \node (2134) at (11,3.5) {$2134$};
        \node (3241) at (13,3.5) {$3241$};
        \node (2341) at (15,3.5) {$2341$};
        \node (2431) at (17,3.5) {$2431$};
        \node (2413) at (19,3.5) {$2413$};
        \draw [green,thick] (1423)--(1324);
        \draw [violet,thick] (1423)--(1432);
        \draw [violet,thick] (1324)--(1342);
        \draw [brown,thick] (1432)--(1342);
        \draw [red,thick] (1342)--(1234);
        \draw [brown,thick] (1342)--(1243); 
        \draw [red,thick] (1234)--(1243);
        \draw [violet,thick] (2143)--(2134);
        \draw [gray,thick] (2134)--(3241);
        \draw [brown,thick] (2341)--(2431);
        \draw [violet,thick] (2431)--(2413);
        \node (4132) at (0,0) {$4132$};
        \node (4123) at (2,1.5) {$4123$};
        \node (4231) at (2,-1.5) {$4231$};
        \node (4213) at (4,0) {$4213$};
        \node (4312) at (6,1.5) {$4312$};
        \node (4321) at (6,-1.5) {$4321$};
        \node (3412) at (9,1.5) {$3412$};
        \node (3421) at (11,1.5) {$3421$};
        \node (2314) at (13,1.5) {$2314$};
        \node (3214) at (15,1.5) {$3214$};
        \node (3124) at (17,1.5) {$3124$};
        \node (3142) at (19,1.5) {$3142$};
        \draw [violet,thick] (4132)--(4123);
        \draw [green,thick] (4132)--(4231);
        \draw [brown,thick] (4123)--(4213);
        \draw [violet,thick] (4231)--(4213);
        \draw [brown,thick] (4213)--(4312); 
        \draw [red,thick] (4213)--(4321);
        \draw [red,thick] (4312)--(4321);
        \draw [violet,thick] (3412)--(3421);
        \draw [gray,thick] (3421)--(2314);
        \draw [brown,thick] (3214)--(3124);
        \draw [violet,thick] (3124)--(3142);
        \draw [dashed,thick] (1423)--(4132);
        \draw [dashed,thick] (1432)--(4123);
        \draw [dashed,thick] (1324)to [bend right=100](4231);
        \draw [dashed,thick] (1342)--(4213);
        \draw [dashed,thick] (1234) to [bend left](4321);
        \draw [dashed,thick] (1243)--(4312);
        \draw [dashed,thick] (2143)--(3412);
        \draw [dashed,thick] (2134)--(3421);
        \draw [dashed,thick] (3241)--(2314);
        \draw [dashed,thick] (2341)--(3214);
        \draw [dashed,thick] (2431)--(3124);
        \draw [dashed,thick] (2413)--(3142);
    \end{tikzpicture}
    \caption{All tied pairs of patterns in $\S_4$. Dashed edges indicate pairs of the form $\sigma\equiv\comp{\sigma}$. The colors of the solid edges refer to the theorem proving that they are tied patterns: \textcolor{red}{red} for Theorem~\ref{thm:iotarho}, \textcolor{brown}{brown} for Theorem~\ref{thm:NoOverlap} with $j=1$ (including Corollary~\ref{cor:rhorho'}), \textcolor{violet}{violet} for Theorem~\ref{thm:NoOverlap} with $j=2$, \textcolor{green}{green} for Theorem~\ref{thm:1k2k-1general}, and \textcolor{gray}{gray} for Theorem~\ref{thm:2134_3241}.}
    \label{fig:TieLength4}
\end{figure}

The results from Section~\ref{Sec:GeneralProbability} (Theorems \ref{thm:iotarho}, \ref{thm:NoOverlap}, \ref{thm:1k2k-1general} and Corollary~\ref{cor:rhorho'}), along with Observation~\ref{obs}, explain all of these ties except for the pair $2134\equiv3241$ (and its complement $3421\equiv2314$). In this section we will give a proof of this relation.
Interestingly, there does not appear to be an obvious bijection between $\Av_n^{2134}(2134,3241)$ and $\Av_n^{3241}(3241,2134)$. Instead, our proof will construct a bijection between different sets, and then apply the principle of inclusion-exclusion.

Following~\cite{DwyerSergi}, for $\sigma,\tau\in\S_k$ and $\pi\in\S_n$, define 
$$\Em(\{\sigma,\tau\},\pi)=\{i:\st(\pi_i\dots\pi_{i+k-1})\in\{\sigma,\tau\}\}.$$ 
This set keeps track of the positions of the occurrences of $\sigma$ and $\tau$ in $\pi$. For any $S\subseteq\Em(\{\sigma,\tau\},\pi)$, we call 
the ordered pair $(\pi,S)$ a \emph{marked permutation}. The  occurrences of $\sigma$ and $\tau$ in positions in $S$ are called \emph{marked occurrences}.
Suppose that the elements of $S$ are $i_1<i_2<\dots<i_s$. We say that $(\pi,S)$ is a \emph{$\sigma,\tau$-tight cluster} (or simply a  \emph{tight cluster} when the patterns are clear from the context) if $i_1=1$, $i_s=n-k+1$, and 
$i_{j+1}-i_j\le k-2$ for all $j$. Note that, in a tight cluster, adjacent marked occurrences are required to overlap in at least two positions, whereas in the usual definition of a cluster (see e.g.~\cite{ElizaldeNoy5}), they could overlap in one position.

\begin{lemma}\label{lem:moccurrences}
For all $S\subseteq[n-3]$, there is a bijection 
\begin{equation}\label{eq:bij}
\{\pi\in\S_n:\Em(\{2134,3241\},\pi)\supseteq S\}\longrightarrow
\{\pi'\in\S_n:\Em(\{2314,3421\},\pi')\supseteq S\}
\end{equation}
such that, if $\pi$ maps to $\pi'$, then
$\pi_1=\pi'_1$ and $\pi_n=\pi'_n$.
Additionally, if $n-3\in S$, then we have
$\st(\pi_{n-3}\pi_{n-2}\pi_{n-1}\pi_n)=2134$ if and only if $\st(\pi'_{n-3}\pi'_{n-2}\pi'_{n-1}\pi'_n)=2314$.
\end{lemma}

\begin{proof}
For the sets in equation~\eqref{eq:bij} to be nonempty, $S$ cannot have two consecutive entries, since the patterns $2134$ and $3241$ are non-$3$-overlapping, and so are the patterns $2314$ and $3421$.

Partition $S$ into maximal arithmetic sequences with difference $2$. In other words, consider the finest partition such that if $j,j+2\in S$ for some $j$, then $j$ and $j+2$ belong to the same block. Suppose that $B=\{j,j+2,j+4,\dots,j+2(t-1)\}$ is one of the blocks, for some $j,t\ge1$.

Let $\pi$ be a permutation in left-hand side of~\eqref{eq:bij}, and let \begin{equation}\label{eq:alpha}
\alpha=\alpha_B=\st(\pi_j\pi_{j+1}\dots\pi_{j+2t+1})\in\S_{2t+2},
\end{equation}
that is, the permutation obtained by restricting to the marked occurrences in positions $B$. Then the pair $(\alpha,\{1,3,\dots,2t-1\})$ is a $2134,3241$-tight cluster.
Note that for occurrences of these patterns to overlap in two positions, the first pattern must be a $3241$, whereas the second can be either $3241$ or $2134$. Thus, the marked occurrences in the above tight cluster must be all $3241$ except for the last one, which can be either $3241$ or $2134$. We claim that, in each of these two cases, the permutation $\alpha$ is uniquely determined. 

Indeed, when all the marked occurrences are $3241$, we must have 
$\alpha_{2i+2}<\alpha_{2i}<\alpha_{2i-1}<\alpha_{2i+1}$
for all $i\in[t]$, and so
\[\alpha_{2t+2}<\alpha_{2t}<\dots<\alpha_2<\alpha_1<\alpha_3<\cdots<\alpha_{2t+1},\]
or equivalently, 
\begin{equation}\label{eq:alpha1}
    \alpha=(t+2)(t+1)(t+3)t(t+4)(t-1)\dots(2t+1)2(2t+2)1.
\end{equation}
When the last occurrence is $2134$, we have $\alpha_{2i+2}<\alpha_{2i}<\alpha_{2i-1}<\alpha_{2i+1}$
for all $i\in[t-1]$, and 
$\alpha_{2t}<\alpha_{2t-1}<\alpha_{2t+1}<\alpha_{2t+2}$, and so
\[\alpha_{2t}<\alpha_{2t-2}<\dots<\alpha_2<\alpha_1<\alpha_3<\cdots<\alpha_{2t+1}<\alpha_{2t+2},\]
or equivalently, 
\begin{equation}\label{eq:alpha2}
\alpha=(t+1)t(t+2)(t-1)(t+3)(t-2)\dots(2t)1(2t+1)(2t+2).
\end{equation}

Similarly, if $\pi'$ is in the right-hand side of~\eqref{eq:bij}, and $\alpha'=\st(\pi'_\ell\pi'_{\ell+1}\dots\pi'_{\ell+2t+1})\in\S_{2t+2}$, the pair $(\alpha',\{1,3,\dots,2t-1\})$ is a $2314,3421$-tight cluster. The marked occurrences must be all $2314$ except for the last one, which can be either $2314$ or $3421$. In the first case, we must have
\begin{equation}\label{eq:alpha'1}
\alpha'=(t+1)(t+2)t(t+3)(t-1)(t+4)\dots2(2t+1)1(2t+2).
\end{equation}
In the second case,
\begin{equation}\label{eq:alpha'2}
\alpha'=(t+2)(t+3)(t+1)(t+4)t(t+5)\dots3(2t+2)21.
\end{equation}

Now we are ready to describe the bijection. Given $\pi$ in the left-hand side of~\eqref{eq:bij}, for each block $B$ in the partition of $S$, we consider the permutation $\alpha$ obtained by restricting to that block. 
If $\alpha$ is given by equation~\eqref{eq:alpha1}, we turn it into 
$\alpha'$ as given by equation~\eqref{eq:alpha'2}, by permuting the corresponding entries of $\pi$; if $\alpha$ is given by equation~\eqref{eq:alpha2}, we turn it into $\alpha'$ as given by equation~\eqref{eq:alpha'1}.  We let $\pi'$ be the resulting permutation after these changes have been applied to each block $B$. Note that this transformation has a simple description in both cases: $\alpha'$ is obtained from $\alpha$ by transposing the entries in positions $2i$ and $2i+1$ for each $i\in[t]$. Thus, $\pi'$ is obtained from $\pi$ by transposing, for each block $B=\{j,j+2,j+4,\dots,j+2(t-1)\}$ as above, the entries in positions $j+2i-1$ and $j+2i$ for each $i\in[t]$.

It is important that, in both cases, the first and the last entry of $\alpha$ are preserved. This property guarantees that the changes applied to the restrictions of $\pi$ to different blocks $B_1$ and $B_2$ are independent. Indeed, since elements in $B_1$ and $B_2$ differ by at least $3$ by construction, the restrictions $\alpha_{B_1}$ and $\alpha_{B_2}$ do not overlap in more than one position. 
The same property also guarantees that $\pi_1=\pi'_1$ and $\pi_n=\pi'_n$.

Finally, if $n-3\in S$, then the last marked occurrence of the last block is $\pi_{n-3}\pi_{n-2}\pi_{n-1}\pi_n$. 
This must be an occurrence of $2134$ or $3241$. In both cases, the entries $\pi_{n-2}$ and $\pi_{n-1}$ are transposed, so $\pi'$ ends with an occurrence of $2314$ or $3421$, respectively. 
\end{proof}

\begin{example}
    Let $n=15$ and $S=\{3,5,7,10,12\}\subseteq[12]$, and let
    $$\pi=7\,8\,\ul{6}\,5\,\ul{9}\,2\,\ol{11}\,1\,12\,\ul{13}\,10\,\ul{14}\,4\,15\,3,$$
    which satisfies $\Em(\{2134,3241\},\pi)\supseteq S$ (the initial positions of marked occurrences of $3241$ are underlined, and those of $2134$ are overlined). Following the proof of Lemma~\ref{lem:moccurrences}, we partition $S$ into two blocks $\{3,5,7\}$ and $\{10,12\}$. Thus, $\pi'$ is obtained from $\pi$ by transposing the entries in positions $(4,5)$, $(6,7)$ and $(8,9)$ for the first block, and $(11,12)$ and $(13,14)$ for the second block, resulting in 
    $$\pi'=7\,8\,\ul{6}\,9\,\ul{5}\,11\,\ul{2}\,12\,1\,\ol{13}\,14\,\ol{10}\,15\,4\,3,$$
    which satisfies $\Em(\{2314,3421\},\pi')\supseteq S$ (the initial positions of marked occurrences of $2314$ are underlined, and those of $3421$ are overlined).
\end{example}

We remark that, in the above example, $\pi'$ has an additional occurrence of $2314$ at the very beginning, which is not marked (i.e., $1\notin S$), whereas $\pi$ did not have any occurrence of 
$2134$ or $3241$ in that position. This highlights the fact that the bijection in Lemma~\ref{lem:moccurrences} only works once we fix a subset $S$ of marked occurrences. This is why we have to resort to inclusion-exclusion to prove the following lemma.

\begin{lemma}\label{lem:avoidboth}
For all $a,b\in[n]$,
	\[\card{\{\pi\in\Av_n^{2134}(2134,3241):\pi_1=a,\pi_n=b\}}=\card{\{\pi'\in\Av_n^{2314}(2314,3421):\pi'_1=a,\pi'_n=b\}}.\]
\end{lemma}

\begin{proof}
For any set $\{n-3\}\subseteq S\subseteq[n-3]$, let 
\begin{align*}
f_=(S)&=\card{\{\pi\in\S_n:\Em(\{2134,3241\},\pi)=S, \pi_1=a,\pi_n=b, \st(\pi_{n-3}\pi_{n-2}\pi_{n-1}\pi_n)=2134\}},\\
f_\ge(S)&=\card{\{\pi\in\S_n:\Em(\{2134,3241\},\pi)\supseteq S, \pi_1=a,\pi_n=b, \st(\pi_{n-3}\pi_{n-2}\pi_{n-1}\pi_n)=2134\}},\\
g_=(S)&=\card{\{\pi'\in\S_n:\Em(\{2314,3421\},\pi')=S, \pi'_1=a,\pi'_n=b, \st(\pi'_{n-3}\pi'_{n-2}\pi'_{n-1}\pi'_n)=2314\}},\\
g_\ge(S)&=\card{\{\pi'\in\S_n:\Em(\{2314,3421\},\pi')\supseteq S, \pi'_1=a,\pi'_n=b, \st(\pi'_{n-3}\pi'_{n-2}\pi'_{n-1}\pi'_n)=2314\}}.
\end{align*}

Since $f_\ge(S)=\sum_{T\supseteq S}f_=(S)$ and $g_\ge(S)=\sum_{T\supseteq S}g_=(S)$, the Principle of Inclusion-Exclusion \cite[Theorem 2.1.1]{EC1} implies that
\[f_=(S)=\sum_{T\supseteq S}(-1)^{\card{T\setminus S}}f_\ge(T) \quad\text{and}\quad g_=(S)=\sum_{T\supseteq S}(-1)^{\card{T\setminus S}}g_\ge(T).\] 
By Lemma~\ref{lem:moccurrences}, we know that $f_\ge(S)=g_\ge(S)$ for all $\{n-3\}\subseteq S\subseteq[n-3]$.
We conclude that
$f_=(S)=g_=(S)$ for all $\{n-3\}\subseteq S\subseteq[n-3]$.
In particular, taking $S=\{n-3\}$, we have $f_=(\{n-3\})=g_=(\{n-3\})$.
Finally, note that
\begin{align*} 
f_=(\{n-3\})&=\card{\{\pi\in\Av_n^{2134}(2134,3241):\pi_1=a,\pi_n=b\}},\\
g_=(\{n-3\})&=\card{\{\pi\in\Av_n^{2314}(2314,3421):\pi_1=a,\pi_n=b\}}.\qedhere
\end{align*}
\end{proof}

We can now finally prove that $2134$ and $3241$ form a tied pair.

\begin{theorem}\label{thm:2134_3241}
We have $$2134\equiv3241.$$
\end{theorem}

\begin{proof}
    Taking complements of the permutations in the right-hand-side of Lemma~\ref{lem:avoidboth}, we see that, for any $a,b\in[n]$, 
  \begin{multline*}
	\card{\{\pi\in\Av_n^{2134}(2134,3241):\pi_1=a,\pi_n=b\}}\\
	=\card{\{\pi'\in\Av_n^{3241}(2134,3241):\pi'_1=n+1-a,\pi'_n=n+1-b\}}.
\end{multline*}  
Summing over all $a,b\in[n]$, we obtain 
\[\card{\Av_n^{2134}(2134,3241)}=\card{\Av_n^{3241}(2134,3241)},\]
and the result now follows from Corollary~\ref{cor:tie}.
\end{proof}


\section{Non-transitivity and a conjectural winning strategy}
\label{Sec:Conjecture}

A surprising feature of the original Penney's game on words is its non-transitivity, meaning that there are cycles in the directed graph of {\em beater} relations, i.e., the graph on binary words of length $k$ where there is an edge from $v$ to $w$ if, among these two words, $v$ is more likely to appear first in a random sequence of coin flips. In fact, every vertex in this graph has incoming edges~\cite{Felix,GuibasOdlyzko}. 

It follows from Tables~\ref{tab:length3} and~\ref{tab:length4} that the permutation analogue of Penney's game is also non-transitive, and that, at least for $3\le k\le 5$, every pattern of length $k$ is beaten by some other pattern, where we say that $\tau$ {\em beats} $\sigma$ if $\Pr(\sigma\prec\tau)<\frac12$. 

Figure~\ref{Fig:3} shows the graph of beater relations among patterns of length~$3$, which has a directed edge from $\tau$ to $\sigma$ if $\tau$ beats $\sigma$. This edge is a solid edge if $\tau$ is the {\em best beater} for $\sigma$, meaning that it minimizes $\Pr(\sigma\prec\tau)$ over all $\tau$ of the same length as~$\sigma$.

\begin{figure}[htb]
\centering
\begin{tikzpicture}[scale=.9]
\node (123) at (0,1) {$123$};
\node (312) at (0,-1) {$312$};
\node (231) at (1.5,0) {$231$};
\node (321) at (3,1) {$321$};
\node (132) at (3,-1) {$132$};
\node (213) at (4.5,0) {$213$};
\draw[->,thick] (123)--(231);
\draw[->,thick] (231)--(312);
\draw[->,thick] (312)--(123);
\draw[->,thick] (321)--(213);
\draw[->,thick] (213)--(132);
\draw[->,thick] (132)--(321);
\draw[->,thick,dotted] (231)--(321);
\draw[->,thick,dotted] (231)-- (132);
\draw[->,thick,dotted] (213)++(0,.4) to [bend right =40] (123);
\draw[->,thick,dotted] (213)++(0,-.4) to [bend left =40] (312);
\end{tikzpicture}
\caption{The beater relations among patterns in $\S_3$. Solid edges indicate best beaters.}
\label{Fig:3}
\end{figure}
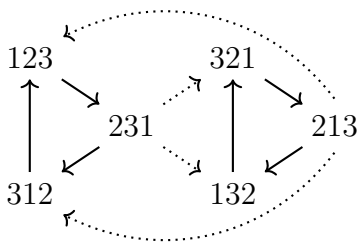

A {\em winning strategy} for player~B is a method that, given any choice of $\sigma\in\S_k$ for player~A, produces a pattern $\tau\in\S_k$ that beats $\sigma$.
For the original Penney's game on words, it was shown by Guibas--Odlyzko~\cite{GuibasOdlyzko} and Felix~\cite{Felix} that, if player~A chooses a word $a_1\dots a_k$, then there is a winning strategy for player~B which consists of choosing a word of the form $ba_1\dots a_{k-1}$ for some suitable $b$. More generally, a description of the values of $b$ that maximize the winning probability for player~B was given in~\cite{Felix} for words over any finite alphabet. We conjecture that a similar winning strategy exists for our permutation version of Penney's game.

\begin{conjecture}\label{conj}
	For any $k\ge3$ and any $\sigma=\sigma_1\dots\sigma_{k-1}\sigma_k\in\S_k$, the permutation $\tau=\sigma_k\sigma_1\dots\sigma_{k-1}$ satisfies that
 \[\Pr(\sigma\prec\tau)<\frac12,\]
 thus giving a winning strategy for player~B.
\end{conjecture}

This conjecture holds for $k=3$, as can be seen from Table~\ref{tab:length3} or Figure~\ref{Fig:3}. 
Based on the numerical estimates obtained by truncating the series in Theorem~\ref{thm:PrGF} at $n=11$, the conjecture also holds for $k=4$ and $k=5$.
In many of the cases we tried, this strategy also seems to give the best beater, but not always. For example, for length~4 we found one exception (two if we consider symmetries): letting $\sigma=2341$, the pattern $\tau=1234$ suggested by Conjecture~\ref{conj} satisfies
 \[\Pr(\sigma\prec\tau)\approx0.443,\]
whereas the best beater is $\tau'=2134$, which satisfies
 \[\Pr(\sigma\prec\tau')\approx0.420.\]
It would be interesting to describe an optimal strategy for player~B.

\begin{problem}\label{prob}
	For any $k\ge3$ and any $\sigma\in\S_k$, find a permutation $\tau$ that minimizes $\Pr(\sigma\prec\tau)$.
\end{problem}
For patterns of length~4, the graph of the best beater relations, obtained using numerical estimates, is drawn in Figure~\ref{Fig:4}.

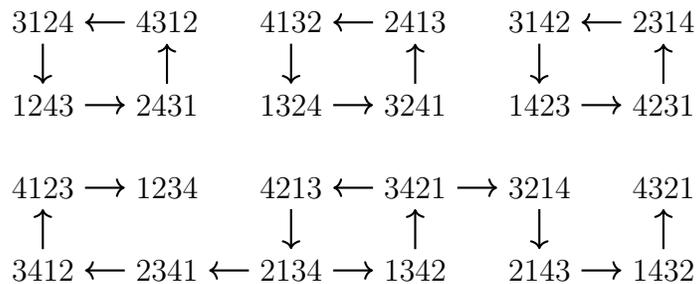
\begin{figure}[htb]
\centering
\begin{tikzpicture}[scale=.55]
\node (1243) at (0,0) {$1243$};
\node (2431) at (3,0) {$2431$};
\node (4312) at (3,2) {$4312$};
\node (3124) at (0,2) {$3124$};
\draw[->,thick] (1243)--(2431);
\draw[->,thick] (2431)--(4312);
\draw[->,thick] (4312)--(3124);
\draw[->,thick] (3124)--(1243);
\begin{scope}[shift={(6,0)}]
\node (1324) at (0,0) {$1324$};
\node (3241) at (3,0) {$3241$};
\node (2413) at (3,2) {$2413$};
\node (4132) at (0,2) {$4132$};
\draw[->,thick] (1324)--(3241);
\draw[->,thick] (3241)--(2413);
\draw[->,thick] (2413)--(4132);
\draw[->,thick] (4132)--(1324);
 \end{scope}
\begin{scope}[shift={(12,0)}]
\node (1423) at (0,0) {$1423$};
\node (4231) at (3,0) {$4231$};
\node (2314) at (3,2) {$2314$};
\node (3142) at (0,2) {$3142$};
\draw[->,thick] (1423)--(4231);
\draw[->,thick] (4231)--(2314);
\draw[->,thick] (2314)--(3142);
\draw[->,thick] (3142)--(1423);
  \end{scope}
\begin{scope}[shift={(6,-4)}]
\node (2134) at (0,0) {$2134$};
\node (1342) at (3,0) {$1342$};
\node (3421) at (3,2) {$3421$};
\node (4213) at (0,2) {$4213$};
\node (2341) at (-3,0) {$2341$};
\node (3412) at (-6,0) {$3412$};
\node (4123) at (-6,2) {$4123$};
\node (1234) at (-3,2) {$1234$};
\node (3214) at (6,2) {$3214$};
\node (2143) at (6,0) {$2143$};
\node (1432) at (9,0) {$1432$};
\node (4321) at (9,2) {$4321$};
\draw[->,thick] (2134)--(1342);
\draw[->,thick] (1342)--(3421);
\draw[->,thick] (3421)--(4213);
\draw[->,thick] (4213)--(2134);
\draw[->,thick] (2134)--(2341);
\draw[->,thick] (2341)--(3412);
\draw[->,thick] (3412)--(4123);
\draw[->,thick] (4123)--(1234);
\draw[->,thick] (3421)--(3214);
\draw[->,thick] (3214)--(2143);
\draw[->,thick] (2143)--(1432);
\draw[->,thick] (1432)--(4321);
 \end{scope}
\end{tikzpicture}
\caption{The best-beater relations among patterns in $\S_4$.}
\label{Fig:4}\end{figure}

We conclude by mentioning several directions for further work. 
In our permutation analogue of Penney's game, we generated permutations by using a sequence $\X$ of i.i.d.\ continuous random variables. However, there are other reasonable models to generate random permutations, which result in different pattern probabilities. For example, one can use a Markov chain to generate each permutation from the previous one by setting the transition probability from $\sigma\in\S_k$ to $\tau\in\S_k$ to be $1/k$ if the last $k-1$ entries of $\sigma$ are in the same relative order as the first $k-1$ entries of $\tau$, and $0$ otherwise. An initial analysis of Penney's game under this model can be found in~\cite{Kathy-thesis}.

Other interesting variations of our permutation analogue of Penney's game arise when using different notions of pattern containment.  One downside of such variations is that, unlike in our version, it is not enough to look at the last $k$ values of the random sequence $\X$ to determine the appearance of a pattern of length $k$. For example, using the classical notion of pattern containment, the game would end when the random sequence $\X$ has any subsequence, not necessarily consecutive, whose entries are in the same relative order as one of the permutations chosen by the players.

Alternatively, one could use a more restrictive notion of pattern containment by requiring both the positions and the values of an occurrence to be consecutive. These are called {\em Hertzsprung} patterns in~\cite{Claesson-H}, {\em rigid} patterns in~\cite{Myers}, and {\em very tight} patterns in~\cite{Bona}. 
Ordinary generating functions for occurrences of such patterns in permutations have been obtained by Claesson~\cite{Claesson-H} by adapting some of the techniques used in the case of words~\cite{GuibasOdlyzko}.
However, a major drawback of Penney's game in this setting is that it is not guaranteed to end. Indeed, the expected time to find the first occurrence of the Hertzsprung pattern $12$ is already infinite. And if $\sigma$ is a Hertzsprung pattern of length at least~3, there is a positive probability that $\sigma$ will never appear at all in the random sequence. Another unusual feature in this setting is that a Hertzsprung pattern that appears in $X_1,X_2,\dots,X_{k}$ may no longer appear in $X_1,X_2,\dots,X_{k+1}$. In particular, the naive generalization of Theorem~\ref{thm:ET} does not work for Hertzsprung patterns.

A hybrid between our permutation version of Penney's game and the original version is obtained by letting the random variables $X_i$ take values in a finite set $[m]$, and allowing patterns to have repeated letters. Our version would then correspond to the case $m\to\infty$.
Finally, one could also consider versions of Penney's game with more than two players.

\acknowledgements
\label{sec:ack}
The authors thank Peter Winkler, Anders Claesson, and some anonymous referees for helpful suggestions. 

\bibliographystyle{plain}
\bibliography{penneys_game}
\label{sec:biblio}

\end{document}